\documentclass[11pt]{amsart}
\usepackage[table]{xcolor}
\usepackage{a4wide, amsmath, amsthm, mathabx, amscd, amsfonts, amssymb, caption,
	enumerate, epsfig, float, graphics, graphicx, hyperref, IEEEtrantools,
	mathrsfs, subcaption, textcomp, tikz, tikz-cd, verbatim, xypic, ytableau}
\usepackage[T1]{fontenc}

\newtheorem{theorem}{Theorem}[section]
\newtheorem{proposition}[theorem]{Proposition}
\newtheorem{conjecture}[theorem]{Conjecture}

\newtheorem{lemma}[theorem]{Lemma}

\newtheorem{definition}[theorem]{Definition}
\newtheorem{construction}[theorem]{Construction}

\theoremstyle{definition}

\newtheorem{remark}[theorem]{Remark}
\newtheorem{example}[theorem]{Example}

\makeatletter

\newcommand{\dashedrightarrow}[1][2pt]{%
  \settowidth{\@tempdima}{$\rightarrow$}\rightarrow
  \makebox[-\@tempdima]{\hskip-1.5ex\color{white}\rule[0.5ex]{#1}{1pt}}
  \phantom{\rightarrow}
}
\makeatother

\def\mut{\operatorname{mut}}
\def\Spec{\operatorname{Spec}}

\def\CC{\mathbb{C}}
\def\QQ{\mathbb{Q}}
\def\RR{\mathbb{R}}
\def\NN{\mathbb{N}}
\def\ZZ{\mathbb{Z}}
\def\PP{\mathbb{P}}

\def\dim{\operatorname{dim}}

\def\Conv{\operatorname{Conv}}

\def\cX{\mathcal{X}}

\def\Sng{\operatorname{Sing}}

\def\mut{\operatorname{mut}}

\def\tto{\dashedrightarrow}
\def\Newt{\operatorname{Newt}}

\def\fraks{\mathfrak{s}}
\def\Trop{\operatorname{Trop}}

\def\hol{\operatorname{hol}}
\def\Gr{\operatorname{Gr}}
\def\QH{\operatorname{QH}}
\def\Fuk{\mathcal{F}}

\def\D{{\bf D}}

\definecolor{darkgreen}{RGB}{0,153,0}
\definecolor{darkred}{RGB}{204,0,0}
\definecolor{darkblue}{RGB}{0,51,204}
\definecolor{red}{RGB}{242,43,29}
\hypersetup{colorlinks,linkcolor={darkblue},citecolor={darkblue},urlcolor={darkblue}}  

\begin{document}

\title{Exotic Lagrangian tori in Grassmannians}

\author{Marco Castronovo}
\address{Rutgers University - Hill Center for the Mathematical Sciences}
\email{marco.castronovo@rutgers.edu}

\thanks{Partially supported by NSF grant DMS 1711070. Any opinions, findings, and conclusions or recommendations expressed in this material are those of the author and do not necessarily reflect the views of the National Science Foundation.}

\begin{abstract}

We describe an iterative construction of Lagrangian tori in the complex
Grassmannian $\Gr(k,n)$, based on the cluster algebra structure of the coordinate
ring of a mirror Landau-Ginzburg model proposed by Marsh-Rietsch \cite{MR}.
Each torus comes with a Laurent polynomial, and local systems controlled by
the $k$-variables Schur polynomials at the $n$-th roots of unity. We use this
data to give examples of monotone Lagrangian tori that are neither displaceable
nor Hamiltonian isotopic to each other, and that support nonzero objects in
different summands of the spectral decomposition of the Fukaya category over $\CC$.

\end{abstract}

\maketitle
\thispagestyle{empty}

\section{Introduction}\label{SecIntro}

\subsection{Lagrangian tori}

The construction and classification of Lagrangian submanifolds is a driving
question in symplectic topology, with Lagrangian tori having a prominent
role. One reason for this is the origin of the field in the
Hamiltonian formulation of classical mechanics. In this context, the Arnold-Liouville
theorem constrains the level sets of a completely integrable system to be
Lagrangian tori; see e.g. Duistermaat \cite{Du}. A more recent motivation is
the geometric description of mirror symmetry, where Lagrangian tori
arise as generic fibers of Strominger-Yau-Zaslow fibrations \cite{SYZ}. Lagrangian
tori are also of interest in low-dimensional topology: the Luttinger surgery \cite{Lu}
operation was used by Auroux-Donaldson-Katzarkov \cite{ADK} to study symplectic
isotopy classes of plane curves; Vidussi \cite{Vid} and Fintushel-Stern \cite{FS}
found connections between Seiberg–Witten invariants and Lagrangian tori.
In general dimension, Lagrangian tori in the standard symplectic $\RR^{2n}$ have
been the subject of much investigation: Viterbo \cite{Vit} and Buhovsky \cite{Bu} constrained
their Maslov class; Chekanov \cite{Ch} classified those of product type \cite{Ch};
Chekanov-Schlenk \cite{CS} and Auroux \cite{Au} constructed examples that are
not products.

\subsection{Disk potentials}

A unifying way to think about these results
is to consider Lagrangian tori $L\subset\RR^{2n}=\CC^n$ as boundary conditions
for maps $u:D^2\to\CC^n$ satisfying the nonlinear Cauchy-Riemann type
equation $\overline{\partial}_J(u)=0$, where $J$ is an almost-complex structure
on the target that may vary from point to point and be non-integrable. One can
try to understand how $J$-holomorphic disks change as $L$ is deformed through
Lagrangian embeddings; many known results focus on deformations by Hamiltonian isotopies.
This line of thought generalizes to the global case, when $L\subset X$ is not in a Darboux chart of the symplectic
manifold $X$; however, $J$-holomorphic disks are not easy to describe
for an arbitrary target $X$. Since the work of Floer \cite{Fl} and Oh \cite{Oh},
the monotone case has been the focus of much investigation. A symplectic manifold
$(X^{2N},\omega)$ is monotone if $[\omega]$ and the first Chern class $c_1(X)$ are
positively proportional in $H^2(X;\RR)$; a Lagrangian $L\subset X$ is monotone if
the area $\omega(\beta)$ of disk classes $\beta\in H_2(X,L;\RR)$ is positively
proportional to their Maslov index $\mu(\beta)$. In this setting, for generic $J$
the moduli space $\mathcal{M}_J(L,\beta)$ of unparametrized $J$-holomorphic
disks with boundary on $L$, homology class $\beta$ and a boundary marked point $\bullet$
is a compact manifold of dimension $\mu(\beta)+\operatorname{dim}(L)-2$. One can encode counts of $J$-holomorphic
disks in a finite generating function called disk potential, and try
to establish general properties of the function that may imply something
about its coefficients. The disk potential of a monotone Lagrangian torus $L^N\subset X^{2N}$
is defined as
$$W_L = \sum_{\beta\in H_2(X,L;\ZZ)}c_\beta(L)x^{\partial\beta}\in\CC[x_1^{\pm},\ldots ,x_{N}^{\pm}] \quad ;$$
here the degree $c_\beta(L)=\operatorname{deg}\left(ev_\bullet:\mathcal{M}_J(L,\beta)\to L\right)\in\ZZ$
of the evaluation map $ev_\bullet$ at the marked point $\bullet\in\partial D^2$ is independent of $J$,
and rigid disks have $\mu(\beta)=2$; monotonicity implies that 
$c_\beta(L)\neq 0$ for finitely many classes $\beta$.
When writing the disk potential, we implicitly assume the choice of a basis
of cycles $\gamma_1,\ldots ,\gamma_N\in H_1(L;\ZZ)\cong\ZZ^N$, so that $J$-holomorphic
disks with boundary of class $\partial\beta = k_1\gamma_1 +\cdots +k_N\gamma_N$
contribute to the monomial $x^{\partial\beta}=x^{k_1}\cdots x^{k_N}$.
A Hamiltonian isotopy $\phi^t$ gives an isomorphism $(\phi^t)_*: H_1(L;\ZZ)\to H_1(\phi^t(L);\ZZ)$,
and $W_{\phi^t(L)}=W_L$ in the induced basis of cycles.
It is known that the critical points of $W_{L}$ obstruct
Hamiltonian displaceability; see Cho-Oh \cite[Proposition 7.2]{CO} for toric
moment fibers, Auroux \cite[Proposition 6.9]{AuT} and Sheridan
\cite[Proposition 4.2]{Sh} for a general discussion. Disk potentials have been used by Vianna \cite{Vi1, Vi2}
to distinguish infinitely many monotone Lagrangian tori in complex surfaces
$X$ of Fano type; see also Pascaleff-Tonkonog \cite{PT}.

\subsection{A cluster construction}

In this article, we construct Lagrangian tori in a class of Fano manifolds
of arbitrarily large dimension: the Grassmannians $\Gr(k,n)$ of complex
$k$-dimensional linear subspaces in $\mathbb{C}^n$.

\begin{construction}\label{ConTori}
Given integers $1\leq k < n$, for any Pl\"{u}cker sequence $\fraks$ of type $(k,n)$
there is a corresponding Lagrangian torus $L_\fraks\subset\Gr(k,n)$, equipped
with a canonical basis of cycles $\gamma_d\in H_1(L_\fraks;\ZZ)$ labeled by
Young diagrams $d\subseteq k\times (n-k)$. The torus comes with a rational function
$W_\fraks$ of formal variables $x_d$.
\end{construction}

The Pl\"{u}cker sequences $\fraks$ are based on the notion of quiver mutation
from representation theory; see Section \ref{SecConstruction} for more details
and Example \ref{ExPluckerSeq25} below.
The Lagrangian tori $L_\fraks$ are obtained from algebraic degenerations
to (singular) toric varieties $\Gr(k,n)\rightsquigarrow X(\Sigma_\fraks)$, using
a general technique for constructing completely integrable systems on complex
projective manifolds studied by Harada-Kaveh \cite{HK}; the notion of toric
degeneration is explained in Section \ref{SecLagrangians}. Such degenerations
of Grassmannians have been studied by Rietsch-Williams \cite{RW} in connection
with the theory of Okounkov bodies \cite{Ok, LM, KK}.
All Pl\"{u}cker sequences start from a single initial
seed, and the rational functions $W_\fraks$ are obtained by explicit rational
changes of variable from a single initial Laurent polynomial $W_0$, whose variables are labeled by
those Young diagrams $d\subseteq k\times(n-k)$ that are rectangles. In \cite{Ca},
it was proved that $W_0$ is in fact the disk potential of the monotone Lagrangian
torus fiber of the Gelfand-Cetlin integrable system introduced by Guillemin-Sternberg \cite{GS}.
The formulation of the construction as iterative procedure is particularly convenient
for computational purposes. To illustrate this point, we created a random walk that
generates Pl\"{u}cker sequences $\fraks$ of arbitrary length, and computes the
corresponding Laurent polynomials $W_\fraks$ explicitly; the code is available for
inspection and experiments \cite{CaCode}.

\subsection{Topology of Laurent/positivity phenomena}

By computing a few examples of $W_\fraks$, one quickly notices the following
two phenomena, which are not a direct consequence of the construction:

\begin{enumerate}
	\item the rational function $W_\fraks$ is a Laurent polynomial ;
	\item the coefficients of each Laurent polynomial $W_\fraks$ are natural numbers .
\end{enumerate}

Property (1) is related to the Laurent phenomenon of cluster algebras, a
notion developed by Fomin-Zelevinsky \cite{FZ}. Think each $x_d$ as a Pl\"{u}cker
coordinate on the dual Grassmannian $\Gr^\vee(k,n)=\Gr(n-k,n)\subset\PP^{{n\choose k}-1}$
in its Pl\"{u}cker embedding; the definition of Pl\"{u}cker coordinate is
recalled in Definition \ref{DefPluckerCoordinate}. Each Pl\"{u}cker sequence
$\fraks$ singles out an open algebraic torus chart $T_\fraks\subset U_{k,n}=\Gr^\vee(k,n)\setminus D^\vee_{FZ}$ in
the complement of a particular divisor $D^\vee_{FZ}$, and the
global regular functions $\mathcal{A}_{k,n}=\mathcal{O}(U_{k,n})$ form a cluster algebra;
see Scott \cite{Sc}.
The space $U_{k,n}$ is a smooth affine variety known as open positroid stratum, and
its properties have been the focus of several works in representation theory,
combinatorics, topology and mirror symmetry \cite{Lus, Po, KLS, STWZ, RW}.
By a result of Marsh-Rietsch \cite{MR}, one can think of each
$W_\fraks$ as restriction $W_{\fraks}=W_{\vert T_\fraks}$ of a single global
regular function $W\in\mathcal{A}_{k,n}$ called Landau-Ginzburg potential.
Property (2) is related to positivity of cluster algebras, which has been proved
by Gross-Hacking-Keel-Kontsevich \cite{GHKK}. Their proof consists in interpreting
the coefficients of certain elements of a cluster algebra as generating functions
of tropical curves called broken lines. In mirror symmetry, broken lines
are expected to correspond to the $J$-holomorphic disks of symplectic topology,
and this heuristic leads us to the following.

\begin{conjecture}(see the more precise Conjecture \ref{ConjPotentials})\label{ConjIntro}
The Laurent polynomial $W_\fraks$ is an invariant of the Hamiltonian isotopy class
of the Lagrangian torus $L_\fraks\subset\Gr(k,n)$.
\end{conjecture}

\ytableausetup{boxsize=0.3em}
\begin{figure}[H]
  \centering
    \begin{subfigure}[b]{0.3\textwidth}
        \includegraphics[width=\textwidth]{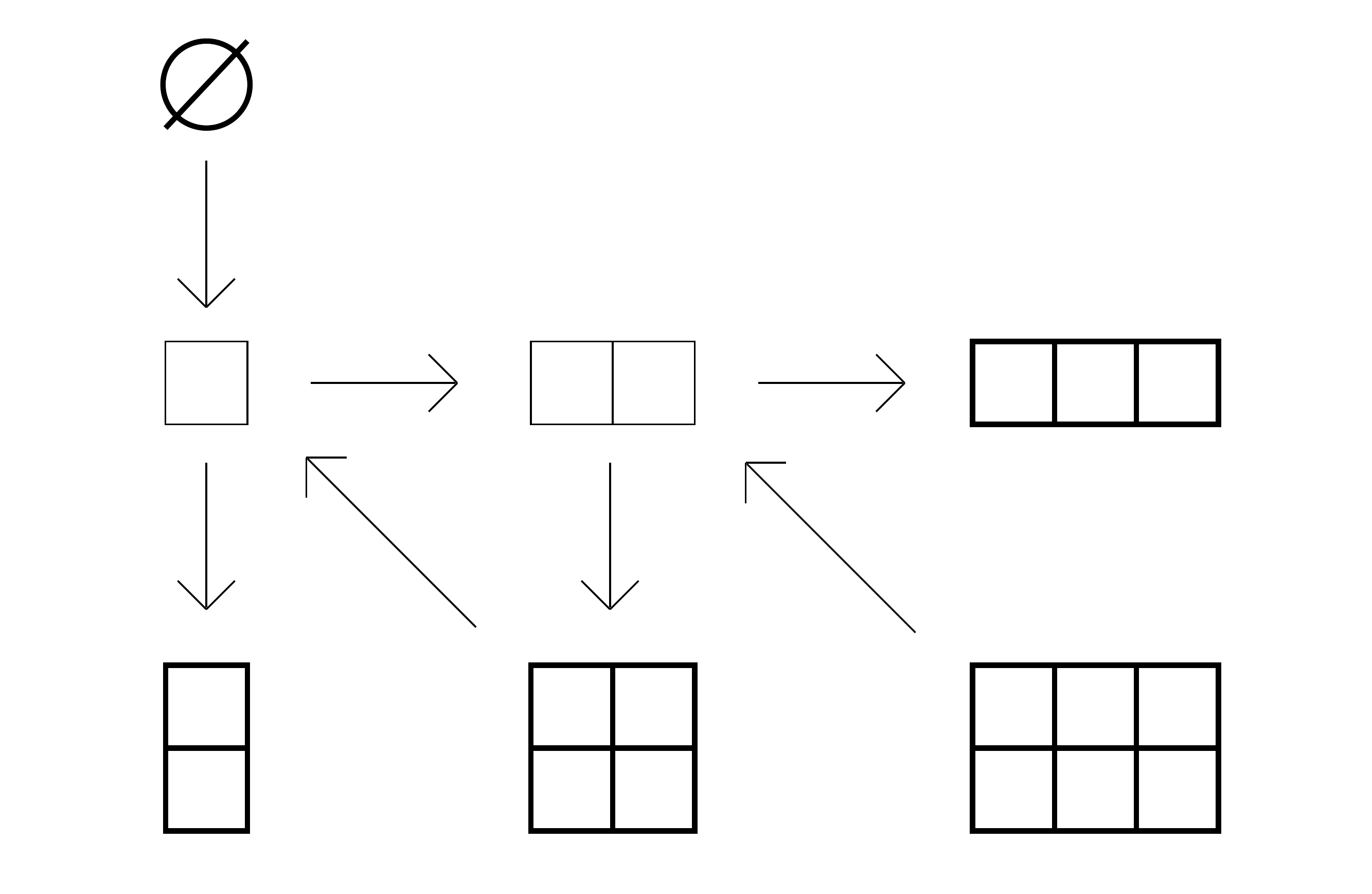}
        \caption{$Q_0$, mutation at $v=\ydiagram{1,0}$}
    \end{subfigure}
    \begin{subfigure}[b]{0.3\textwidth}
        \includegraphics[width=\textwidth]{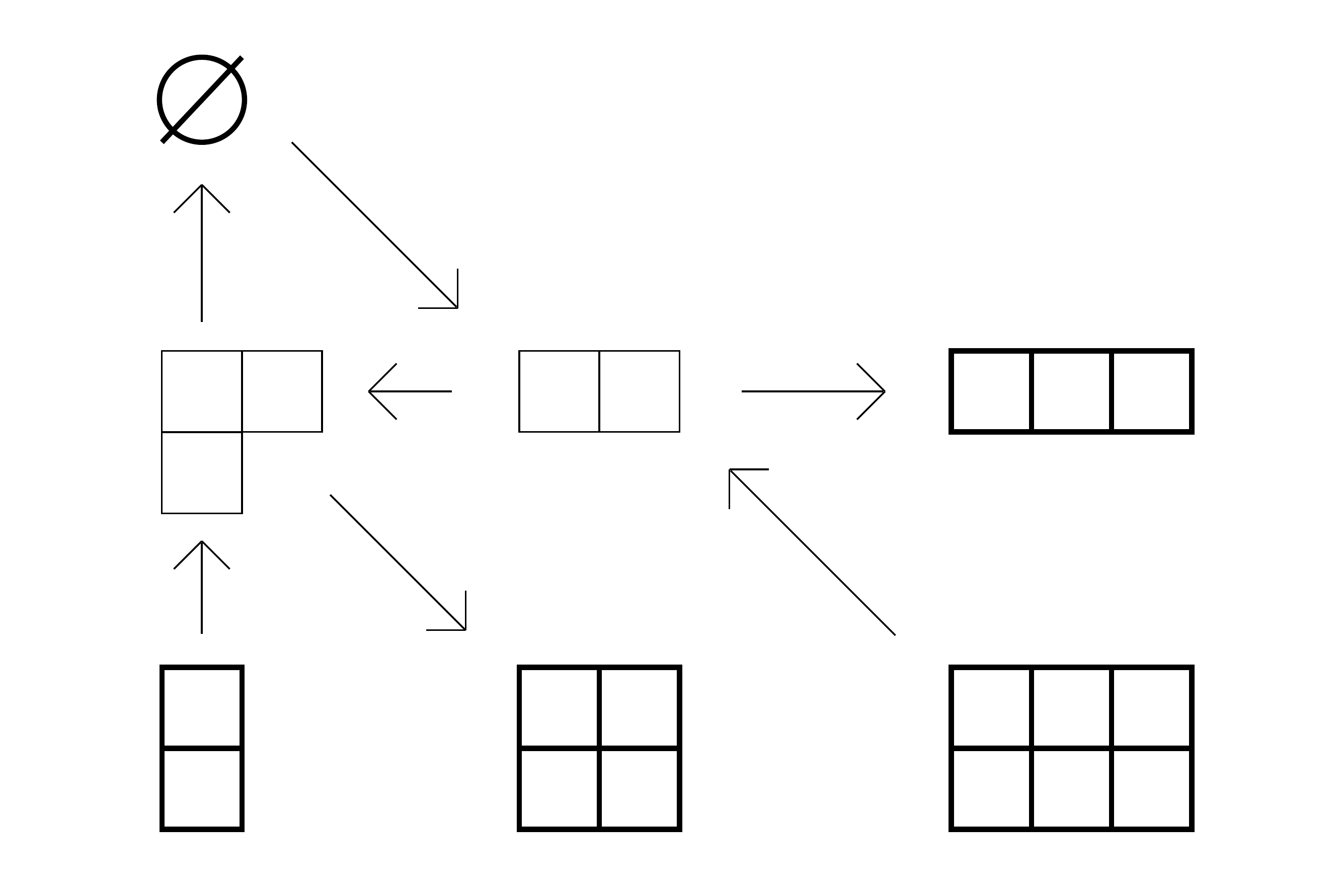}
        \caption{$Q_1$, mutation at $v=\ydiagram{2,0}$}
    \end{subfigure}
    \begin{subfigure}[b]{0.3\textwidth}
        \includegraphics[width=\textwidth]{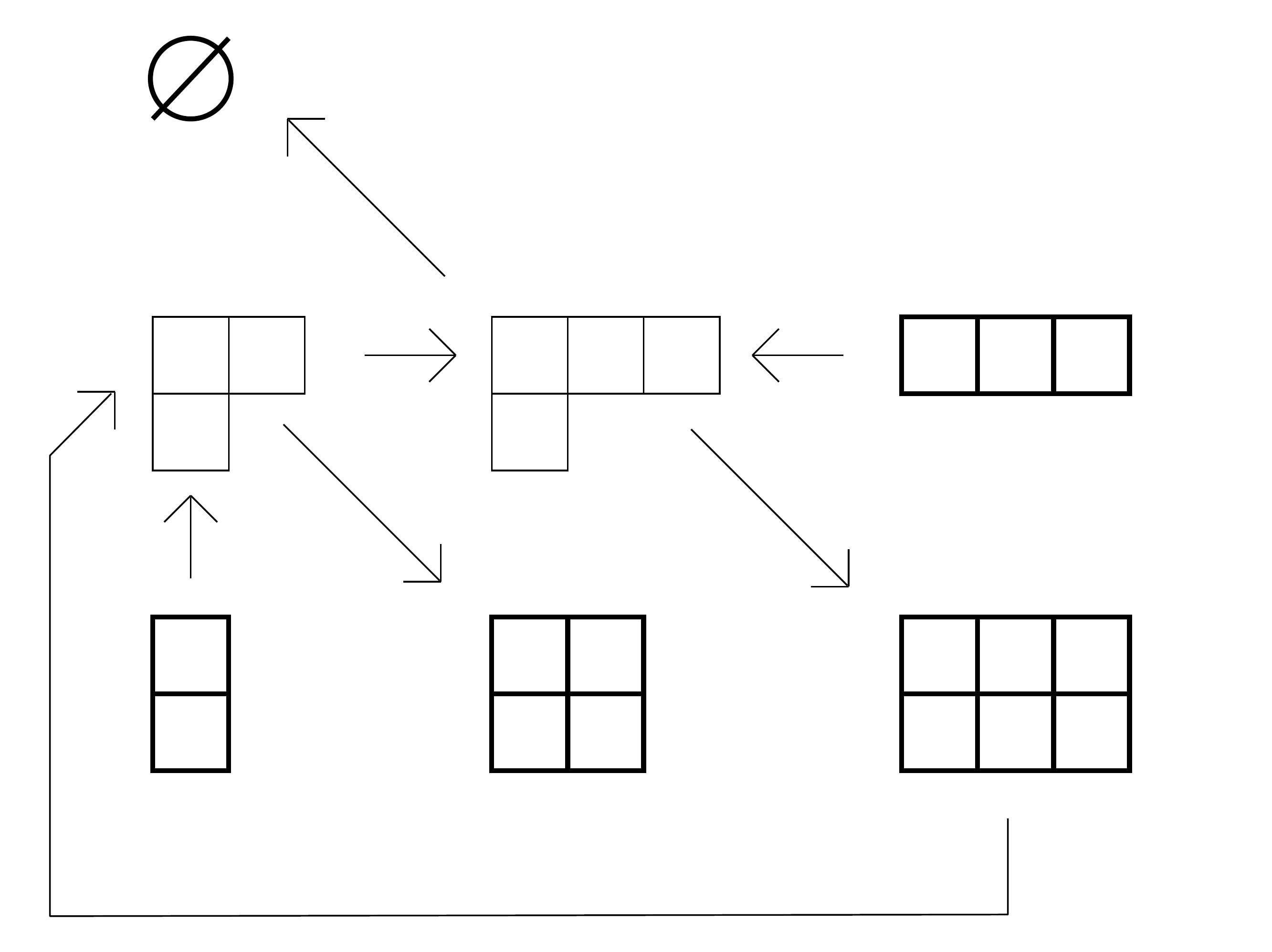}
        \caption{$Q_2$, mutation at $v=\ydiagram{2,1}$}
    \end{subfigure}
    
    \begin{subfigure}[b]{0.3\textwidth}
        \includegraphics[width=\textwidth]{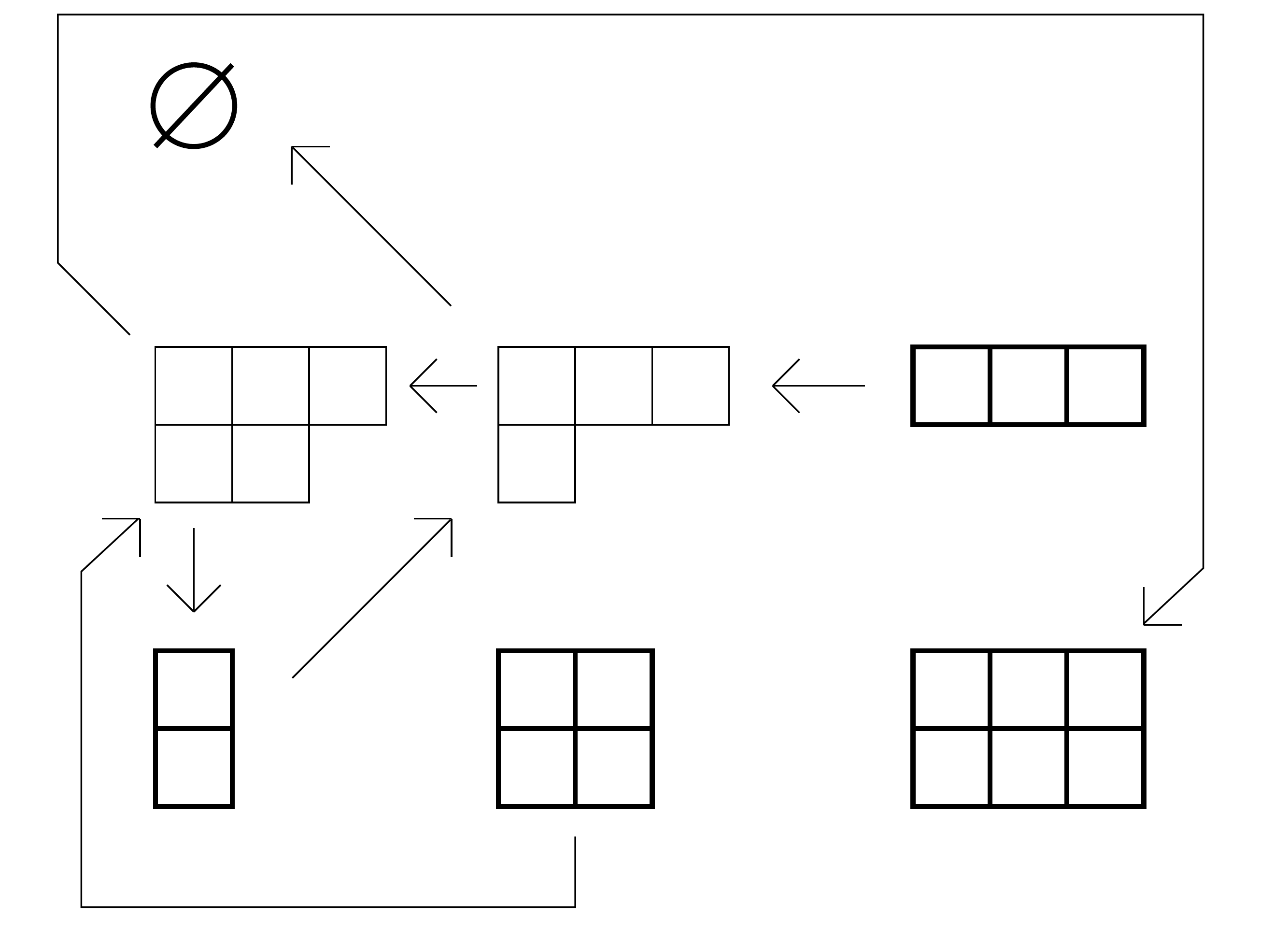}
        \caption{$Q_3$, mutation at $v=\ydiagram{3,1}$}
    \end{subfigure}
    \begin{subfigure}[b]{0.3\textwidth}
        \includegraphics[width=\textwidth]{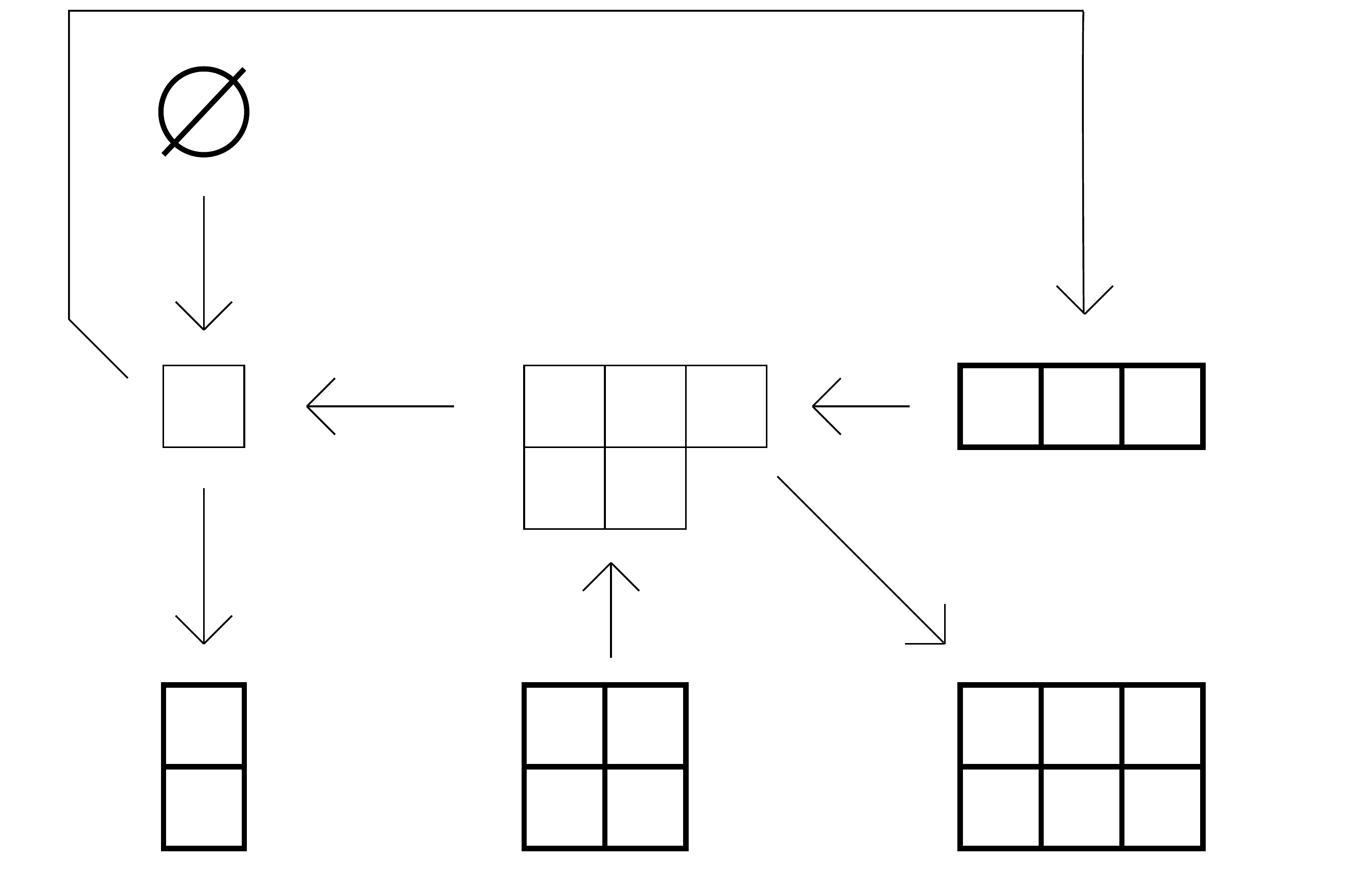}
        \caption{$Q_4$, mutation at $v=\ydiagram{3,2}$}
    \end{subfigure}
    \begin{subfigure}[b]{0.3\textwidth}
        \includegraphics[width=\textwidth]{Q0}
        \caption{$Q_5=Q_0$}
    \end{subfigure}

    \caption{A Pl\"{u}cker sequence of type $(2,5)$ and length five. The labeling variables $x_d$
    on the nodes are replaced  by $d$ for notational convenience.}
    \label{FigPluckerSeq25}
\end{figure}

\begin{example}\label{ExPluckerSeq25}
Let $k=2$ and $n=5$. Figure \ref{FigPluckerSeq25} represents a Pl\"{u}cker sequence
$$\fraks : (Q_0,W_0) \to (Q_1,W_1) \to (Q_2,W_2) \to (Q_3,W_3) \to (Q_4,W_4) \to (Q_5,W_5)$$
of type $(2,5)$ and length five, where the final step and the initial
one coincide. At each step $(Q_i,W_i)$, the graph
$Q_i$ is a quiver whose nodes are labeled by Pl\"{u}cker coordinates $x_d$ for
some collection of Young diagrams $d\subseteq 2\times 3$, and $W_i$ is a Laurent
polynomial of the variables $x_d$. A step $(Q_i,W_i)\to(Q_{i+1},W_{i+1})$
in the sequence consists in performing a quiver mutation at a mutable node $v$
of $Q_i$ as described in Section \ref{SecConstruction}. This procedure changes
the label $l(v)$ of the node $v$ in $Q_i$ to a new label $l'(v)$ of the same node
in $Q_{i+1}$. The two labels are related by
the following exchange relation: $l(v)l(v')$ is a sum of two terms, obtained by taking the product
of labels $l(w)$ from incoming/outgoing nodes $w$ adjacent to $v$ respectively.
The rational function $W_{i+1}$ is obtained from $W_i$ by using the previous
relation to replace the label $l(v)$ with $l(v')$, and becomes Laurent
modulo Pl\"{u}cker relations, i.e. when interpreted as element of the function
field $\operatorname{Frac}(\mathcal{A}_{2,5})=\CC(U_{2,5})=\CC(\Gr^\vee(2,5))$
of the dual Grassmannian $\Gr^\vee(2,5)=\Gr(3,5)$. The intermediate steps of
Construction \ref{ConTori} produce Lagrangian tori $L_0,L_1,L_2,L_3,L_4\subset\Gr(2,5)$. In this case,
all the tori are monotone and the Laurent polynomials $W_i$ for $0\leq i\leq 4$ 
match their disk potentials $W_{L_i}$:

\ytableausetup{boxsize=0.3em}
\begin{eqnarray*}
&& W_0 =
x_{\ydiagram{1,0}}
+ \frac{x_{\ydiagram{1,1}}}{x_{\ydiagram{1,0}}}
+ \frac{x_{\ydiagram{2,2}}x_\emptyset}{x_{\ydiagram{1,0}}x_{\ydiagram{2,0}}}
+ \frac{x_{\ydiagram{3,3}}x_\emptyset}{x_{\ydiagram{2,0}}x_{\ydiagram{3,0}}}
+ \frac{x_{\ydiagram{2,0}}}{x_{\ydiagram{3,3}}}
+ \frac{x_{\ydiagram{2,0}}}{x_{\ydiagram{1,0}}}
+ \frac{x_{\ydiagram{3,0}}}{x_{\ydiagram{2,0}}}
+ \frac{x_{\ydiagram{2,2}}x_\emptyset}{x_{\ydiagram{1,0}}x_{\ydiagram{1,1}}}
+ \frac{x_{\ydiagram{3,3}}x_{\ydiagram{1,0}}}{x_{\ydiagram{2,0}}x_{\ydiagram{2,2}}}
\quad ;\\
&& W_1 =
\frac{x_{\emptyset}x_{\ydiagram{2,2}}}{x_{\ydiagram{2,1}}}
+ \frac{x_{\ydiagram{2,0}}x_{\ydiagram{1,1}}}{x_{\ydiagram{2,1}}}
+ \frac{x_{\ydiagram{3,3}}x_{\emptyset}}{x_{\ydiagram{2,0}}x_{\ydiagram{2,1}}}
+ \frac{x_{\ydiagram{3,3}}x_{\ydiagram{1,1}}}{x_{\ydiagram{2,2}}x_{\ydiagram{2,1}}}
+ \frac{x_{\ydiagram{2,1}}}{x_{\ydiagram{2,0}}}
+ \frac{x_{\ydiagram{2,1}}}{x_{\ydiagram{1,1}}}
+ \frac{x_{\ydiagram{3,3}}x_{\emptyset}}{x_{\ydiagram{2,0}}x_{\ydiagram{3,0}}}
+ \frac{x_{\ydiagram{2,0}}}{x_{\ydiagram{3,3}}}
+ \frac{x_{\ydiagram{3,0}}}{x_{\ydiagram{2,0}}}
\quad ;\\
&& W_2 =
\frac{x_\emptyset x_{\ydiagram{3,3}}x_{\ydiagram{1,1}}}{x_{\ydiagram{3,1}}x_{\ydiagram{2,1}}}
+ \frac{x_{\emptyset}}{x_{\ydiagram{3,1}}}
+ \frac{x_{\ydiagram{3,0}}x_{\ydiagram{1,1}}}{x_{\ydiagram{3,1}}}
+ \frac{x_{\ydiagram{2,1}}x_{\ydiagram{3,0}}}{x_{\ydiagram{3,1}}x_{\ydiagram{3,3}}}
+ \frac{x_{\ydiagram{3,1}}}{x_{\ydiagram{2,1}}}
+ \frac{x_{\ydiagram{3,1}}}{x_{\ydiagram{3,0}}}
+ \frac{x_{\emptyset}x_{\ydiagram{2,2}}}{x_{\ydiagram{2,1}}}
+ \frac{x_{\ydiagram{3,3}}x_{\ydiagram{1,1}}}{x_{\ydiagram{2,2}}x_{\ydiagram{2,1}}}
+ \frac{x_{\ydiagram{2,1}}}{x_{\ydiagram{1,1}}}
\quad ;\\
&& W_3 = 
\frac{x_{\emptyset}}{x_{\ydiagram{3,1}}}
+ \frac{x_{\ydiagram{3,0}}x_{\ydiagram{1,1}}}{x_{\ydiagram{3,1}}}
+ \frac{x_{\ydiagram{3,1}}}{x_{\ydiagram{3,0}}}
+ \frac{x_{\ydiagram{3,2}}x_{\emptyset}}{x_{\ydiagram{3,1}}}
+ \frac{x_{\ydiagram{3,2}}}{x_{\ydiagram{2,2}}}
+ \frac{x_{\ydiagram{1,1}}x_{\ydiagram{3,0}}}{x_{\ydiagram{3,2}}x_{\ydiagram{3,1}}}
+ \frac{x_{\ydiagram{3,3}}}{x_{\ydiagram{3,2}}}
+ \frac{x_{\ydiagram{2,2}}x_{\ydiagram{3,0}}}{x_{\ydiagram{3,2}}x_{\ydiagram{3,3}}}
+ \frac{x_{\ydiagram{3,1}}x_{\ydiagram{2,2}}}{x_{\ydiagram{3,2}}x_{\ydiagram{1,1}}}
\quad ;\\
&& W_4 =
\frac{x_{\ydiagram{3,2}}}{x_{\ydiagram{2,2}}}
+ \frac{x_{\ydiagram{3,3}}}{x_{\ydiagram{3,2}}}
+ \frac{x_{\ydiagram{2,2}}x_{\ydiagram{3,0}}}{x_{\ydiagram{3,2}}x_{\ydiagram{3,3}}}
+ \frac{x_{\ydiagram{1,0}}}{x_{\ydiagram{3,2}}}
+ x_{\ydiagram{1,0}}
+ \frac{x_{\emptyset}x_{\ydiagram{3,2}}}{x_{\ydiagram{1,0}}x_{\ydiagram{3,0}}}
+ \frac{x_{\emptyset}x_{\ydiagram{2,2}}}{x_{\ydiagram{1,0}}x_{\ydiagram{1,1}}}
+ \frac{x_{\ydiagram{1,1}}}{x_{\ydiagram{1,0}}}
+ \frac{x_{\ydiagram{3,0}}x_{\ydiagram{2,2}}}{x_{\ydiagram{1,0}}x_{\ydiagram{3,2}}}
\quad .
\end{eqnarray*}
\end{example}

The equality $W_\fraks = W_{L_\fraks}$ is an application of a general result of
Nishinou-Nohara-Ueda \cite{NNU} on the behavior of disk potentials under toric
degeneration, which also implies monotonicity of $L_\fraks$;
see Proposition \ref{PropSmallMonotone}. This result gives a sufficient condition
for the equality $W_\fraks = W_{L_\fraks}$, which is the existence of a small toric
resolution for the singular toric variety $X(\Sigma_\fraks)$; see Definition \ref{DefSmallResolution}. Due
to the combinatorial nature of toric varieties, for any given Pl\"{u}cker sequence $\fraks$
one can check this condition in finitely many steps. In Section \ref{SecApplications} we use this to
describe a sample application in the smallest example not accessible by previous techniques.

\begin{theorem}(see Theorem \ref{ThmExoticTori})
The Grassmannian $\Gr(3,6)$ contains at least $6$ monotone Lagrangian tori
that are non-displaceable and pairwise inequivalent under Hamiltonian isotopy.
\end{theorem}

We call these tori exotic, because only one monotone torus was previously known:
the Gelfand-Cetlin torus. The new examples are of the form $L_\fraks$ for some Pl\"{u}cker sequence
$\fraks$, and are distinguished by a combination of two invariants:
the number of critical points of their disk potential $W_{L_\fraks}$ and the $f$-vector
of its Newton polytope. This strategy applies without modification to arbitrary
Grassmannians. If Conjecture \ref{ConjPotentials} holds, the same arguments of
Theorem \ref{ThmExoticTori} imply that the tori $L_\fraks\subset\Gr(k,n)$ are
always nondisplaceable, and generally not Hamiltonian isotopic. Note that Conjecture
\ref{ConjPotentials} may still hold when the toric variety $X(\Sigma_\fraks)$
has no small toric resolution, and the result of Nishinou-Nohara-Ueda \cite{NNU}
does not apply. In this case, positivity of the coefficients of $W_\fraks$
suggests an enumerative interpretation in terms of counts of
$J$-holomorphic disks with boundary on $L_\fraks$. We plan to explore this
in a separate work, simply pointing out here a possible interpretation in terms
of low-area Floer theory in the sense of Tonkonog-Vianna \cite{TV}.
For $k=1$ one has projective spaces $\Gr(1,n)=\PP^{n-1}$, and there is
only one Pl\"{u}cker sequence $\fraks$ of lenght $0$; in this case $W_\fraks$
is the disk potential of the Clifford torus. In particular, Construction \ref{ConTori} does
not imply the existence of exotic tori in $\PP^2$ established by Vianna \cite{Vi1, Vi2}.
For $k=2$ Construction \ref{ConTori} recovers a different one studied by
Nohara-Ueda \cite{NU20}, who introduced a collection of Lagrangian tori in $\Gr(2,n)$
corresponding to triangulations of an $n$-gon; the relation is explained in
Lemma \ref{LemmaK2}, and it implies that Conjecture \ref{ConjPotentials} holds when $k=2$.

\subsection{Probing the spectral decomposition}

Since for $k=2$ all tori $L_\fraks\subset\Gr(2,n)$ are monotone, it is natural
to think of them as objects of the monotone Fukaya category.
As described by Sheridan \cite{Sh}, the Fukaya category of a monotone symplectic
manifold $X$ has a spectral decomposition 
$$\Fuk(X)=\bigoplus_{\lambda}\Fuk_\lambda(X) \quad .$$
The summands are $A_\infty$-categories indexed by the eigenvalues $\lambda$ of
the operator $c_1\star$ of multiplication by the first Chern class acting on the small
quantum cohomology. The objects of the $\lambda$-summand
are monotone Lagrangians $L_\xi$ equipped with a rank one local system $\xi$
such that
$$m^0(L_\xi)=\sum_\beta c_\beta(L)\hol_\xi(\partial\beta)=\lambda \quad .$$
Definition \ref{DefLocalSystems} introduces some natural local systems
supported on the Lagrangian tori $L_\fraks\subset\Gr(k,n)$, that are
controlled by the values of $k$-variables Schur polynomials at certain roots of unity.
These local systems generalize the ones studied in \cite{Ca} for the
Gelfand-Cetlin torus, that were controlled by Schur polynomials
corresponding rectangular Young diagrams. When $k=2$, we show that the corresponding
objects split-generate the derived Fukaya category $\D\Fuk(\Gr(2,n))$ in some cases,
notably including examples where the Gelfand-Cetlin torus alone fails to do so.

\begin{theorem}(see Theorem \ref{ThmDyadicHMS})
If $n=2^t+1$ for some $t\in\NN^+$, the derived Fukaya category $\D\Fuk(\Gr(2,2^t+1))$
is split-generated by objects supported on a single Pl\"{u}cker torus.
\end{theorem}

The Lagrangian torus in the statement is associated to a special special
triangulation of the $n$-gon, that we call dyadic. In fact, Section \ref{SecApplications}
contains a criterion to prove split-generation of $\D\Fuk(\Gr(2,n))$ by
objects supported on any number of tori $L_\fraks\subset\Gr(2,n)$, whenever $n$ is odd.
The criterion is based on a construction of triangulations of the $n$-gon
whose sides lengths avoid the prime numbers appearing in the factorization of $n$.

\begin{theorem}(see Theorem \ref{ThmPavoidingHMS})
Let $n>2$ be odd, and consider its prime factorization $n=p_1^{e_1}\cdots p_l^{e_l}$.
If for all $1\leq i\leq l$ there exists a triangulation $\Gamma_i$ of $[n]$
that is $p_i$-avoiding, then $\D\Fuk(\Gr(2,n))$ is split generated by objects
supported on $l$ Pl\"{u}cker tori.
\end{theorem}

\begin{remark}
A standard consequence of split-generation is that any monotone Lagrangian
supporting nonzero objects of the Fukaya category must intersect the generator.
\end{remark}

These results seem to suggest that objects supported on the tori $L_\fraks$
could split-generate $\D\Fuk(\Gr(k,n))$ in general, although split-generation
over $\CC$ has a subtle relation with the location of the critical points
of $W\in\mathcal{A}_{k,n}$ relative to the torus charts $T_\fraks\subset U_{k,n}$.
For example, split-generation over $\CC$ fails for $\Gr(2,4)$, where
$W$ has two critical points in a complex codimension $2$ locus of $U_{k,n}$ which is
not covered by cluster charts. We plan to investigate in a separate work how the
situation changes when considering bulk-deformations in the sense of Fukaya-Oh-Ohta-Ono \cite{FOOO}.

\subsection{Mirror symmetry and abundance of Lagrangian tori}

This article can be thought of as part of a broader program aimed
at investigating the abundance of Lagrangian tori in Fano manifolds $X$ with
an anti-canonical divisor $D\subset X$ whose complement $U=X\setminus D$
is a cluster variety. This class includes many homogeneous varieties $X=G/P$ with
$P\subset G$ parabolic subgroup of a complex linear algebraic group. The cluster
variety $U$ comes with a Langlands dual cluster variety $U^\vee$, and
Gross-Hacking-Keel-Kontsevich \cite{GHKK} proposed that $(X,D)$ has a Landau-Ginzburg
model $(U^\vee,W)$ in the sense of homological mirror symmetry. Here $W\in\mathcal{O}(U^\vee)$
is a regular function intrinsically defined by the cluster structure and given
as a sum of theta functions, which are generating functions of discrete objects
called broken lines in a scattering diagram. We expect that the cluster charts
of $U^\vee$ will correspond to certain Lagrangian tori in $L\subset G/P$, and that
the restriction of $W$ to different cluster charts will fully determine their disk
potential $W_L$ in some cases, and in general suffice to distinguish many of their
Hamiltonian isotopy classes in the spirit of Conjecture \ref{ConjIntro}.

\subsection{Algebraic and topological wall-crossing}

When $W_{L_\fraks}=W_\fraks$, the Lagrangian tori $L_\fraks\subset\Gr(k,n)$ constructed
in this article have disk potentials related by algebraic wall-crossing formulas
by construction. It is natural to ask if these formulas correspond
to a topological wall-crossing, i.e. if the tori $L_\fraks$ are connected by
families of Lagrangian immersions that bound Maslov 0 $J$-holomorphic disks
at some intermediate time. We do not investigate this question here, but only
point out that it would be interesting to see if there is a relation between
our examples and the model of Lagrangian mutation studied by Pascaleff-Tonkonog \cite{PT}.

\vspace{0.3cm}

\textbf{Acknowledgements} I thank my PhD advisor Chris Woodward for his constant
encouragement and useful conversations. I also thank Mohammed Abouzaid for useful
conversations about mirror symmetry, Lev Borisov for useful remarks on toric resolutions
of singularities, and Lauren Williams for pointing out that only quiver mutations at 4-valent
nodes are allowed as transition between plabic cluster charts, which was crucial
at some point of the project.

\section{The iterative construction}\label{SecConstruction}

Throughout this article, $k$ and $n$ are integers with $1\leq k < n$. The symbol
$d$ denotes a Young diagram in the $k\times (n-k)$ grid, obtained by
placing $d_i$ consecutive boxes in the $i$-th row for all $1\leq i\leq k$, starting
from the left in each row and with $d_1\geq d_2\geq \cdots \geq d_k$. Chosen $0\leq i\leq k$ and $0\leq j \leq (n-k)$,
one has a rectangular Young diagram $i\times j$, with
$i\times j=\emptyset$ empty diagram if $i=0$ or $j=0$.
A full rank $n\times (n-k)$ matrix $M$ determines an $(n-k)$-dimensional linear
subspace of $\CC^n$ by taking its column-span. If $[M]$ is the equivalence
class of $M$ modulo column operations, write $[M]\in\Gr^\vee(k,n)=\Gr(n-k,n)$ and think of
it as a point of the complex Grassmannian.
Each Young diagram $d$ has a profile path, which connects the top-right corner
of the $k\times (n-k)$ grid to the bottom-left one.
Labeling the steps of the path by $[n]=\{1,\ldots ,n\}$, the vertical steps
of $d$ determine a set $d^|\subset [n]$ with $|d^||=k$, while the horizontal
steps determine a set $d^-\subset [n]$ with $|d^-|=n-k$; see Figure \ref{FigYoungDiagram}
for an example with $k=3$ and $n=8$.

\begin{figure}[H]
  \centering
        \includegraphics[width=0.3\textwidth]{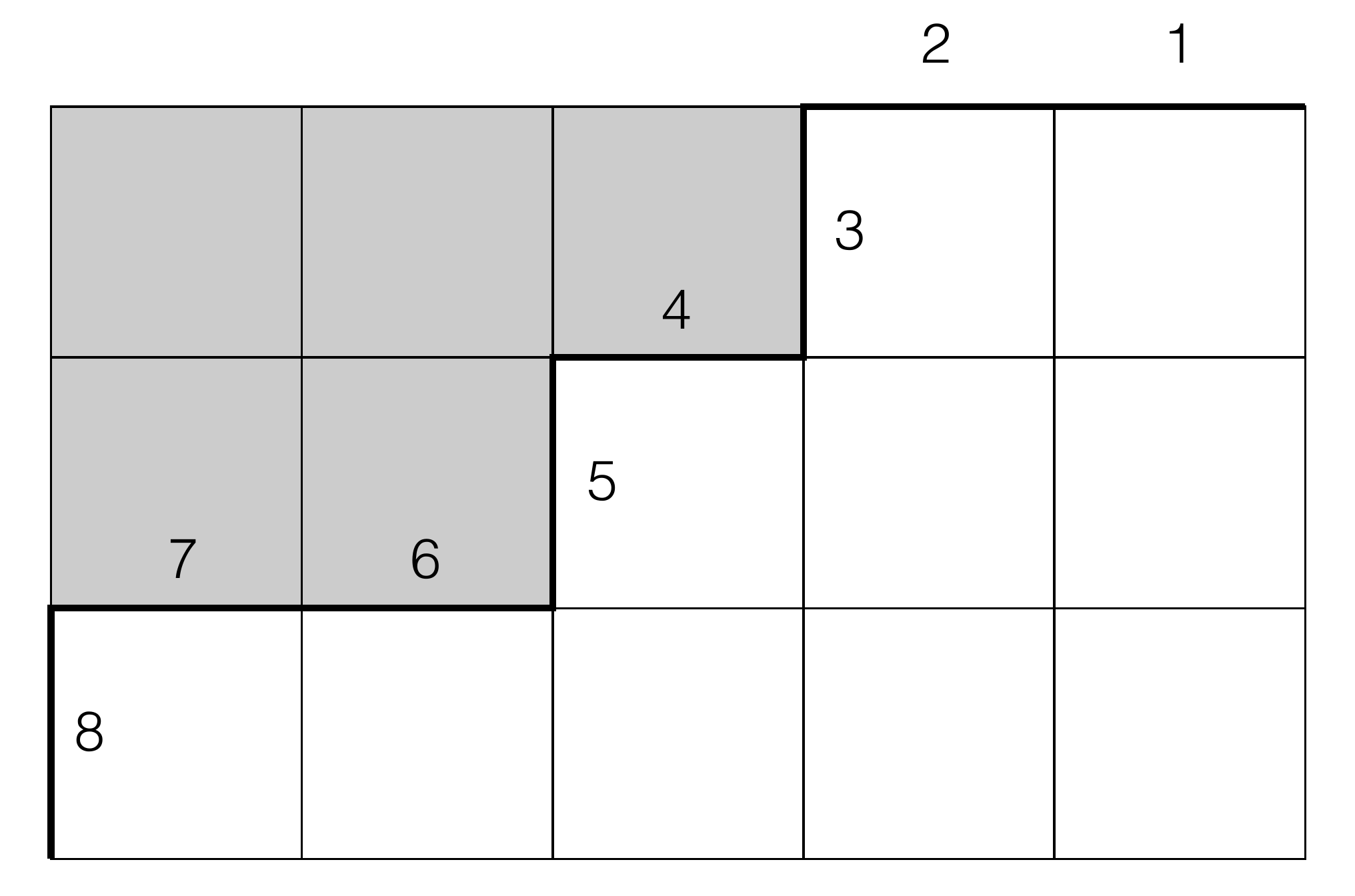}
    \caption{A Young diagram $d\subseteq 3\times 5$ with $d^-=\{1,2,4,6,7\}$ and $d^|=\{3,5,8\}$.}
    \label{FigYoungDiagram}
\end{figure}

\begin{definition}\label{DefPluckerCoordinate}
If $M$ is a full rank $n\times (n-k)$ matrix, the determinant of $M$ at rows
$d^-$ is denoted $x_d(M)$ and called Pl\"{u}cker coordinate corresponding
to $d$.
\end{definition}

The Pl\"{u}cker coordinates define a projective embedding of $\Gr^\vee(k,n)$
in $\PP^{{n\choose k}-1}$. If $\mathcal{I}_{k,n}\subset\CC[x_d \; : \; d\subseteq k\times(n-k)]$
is the corresponding homogeneous ideal, each $x_d$ is an element
of the algebra $\mathcal{A}_{k,n}=\CC[x_d \; : \; d\subseteq k\times(n-k)]/\mathcal{I}_{k,n}$
of regular functions of the affine cone over $\Gr^\vee(k,n)$.

\subsection{Initial seed}
 
\begin{definition}\label{DefQuiverPotential}
A quiver with potential of type $(k,n)$ is a pair $(Q,W)$, where:
\begin{enumerate}
	\item $Q$ is an oriented connected graph, with no edge connecting a node to itself
		and no oriented loops with two edges, whose nodes are labeled by Pl\"{u}cker coordinates $x_d\in\mathcal{A}_{k,n}$ ;
	\item $W$ is a Laurent polynomial in the labels of the nodes of $Q$ .
\end{enumerate}
As part of the data, the nodes of $Q$ are partitioned in two
groups, called frozen and mutable.
\end{definition}

\begin{remark}
To avoid confusion, we point out that Definition \ref{DefQuiverPotential} is
not a special case of the notion of quiver with potential in representation theory: although
$Q$ is a quiver in the classical sense, the potential $W$ is an element of the
commutative algebra $\mathcal{A}_{k,n}$ as opposed to the non-commutative path algebra of $Q$.
\end{remark}

The iterative construction we describe in this section begins with a specific
quiver with potential.

\begin{definition}\label{DefInitialSeed}
The initial seed of type $(k,n)$ is the quiver with potential $(Q_0,W_0)$, where:
\begin{enumerate}
	\item $Q_0$ is the oriented labeled graph in Figure \ref{FigInitialQuiver} ;
	\item $W_0$ is the Laurent polynomial
	$$x_{1\times 1} +
	\sum_{i=2}^{k}\sum_{j=1}^{n-k}\frac{   x_{i\times j}x_{(i-2)\times (j-1)}   }{   x_{(i-1)\times (j-1)}x_{(i-1)\times j}   } +
	\frac{   x_{(k-1)\times (n-k-1)}   }{   x_{k\times (n-k)}   } +
	\sum_{i=1}^{k}\sum_{j=2}^{n-k}\frac{ x_{i\times j}x_{(i-1)\times (j-2)} }{ x_{(i-1)\times (j-1)}x_{i\times (j-1)} }\quad .$$
\end{enumerate}
A node of $Q_0$ is frozen if its label is $x_{i\times j}$ with
$i\times j=\emptyset$, $i=k$ or $j=n-k$; the remaining nodes are mutable.
\end{definition}

\begin{figure}[H]
  \centering
        \includegraphics[width=0.4\textwidth]{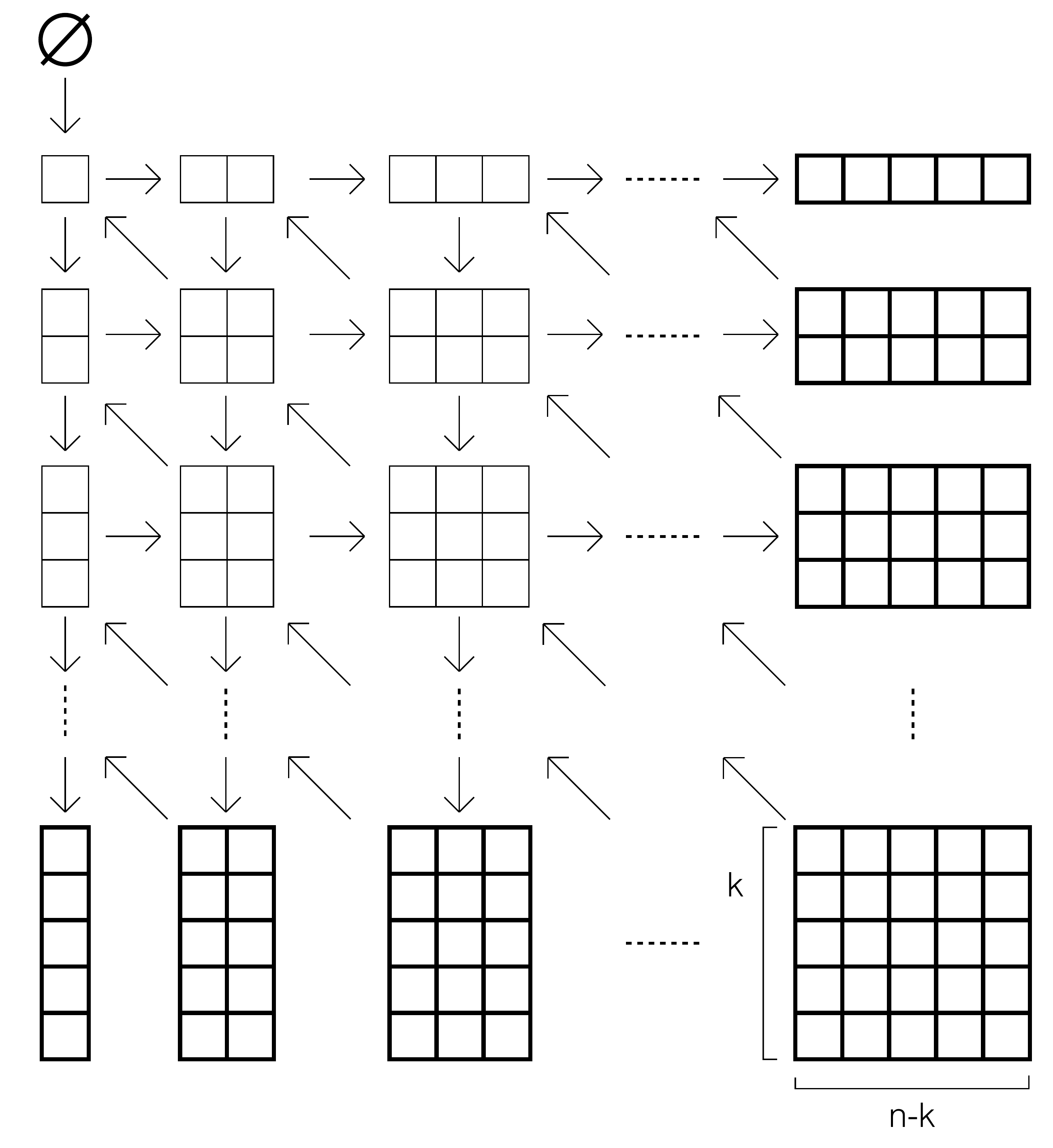}
    \caption{Initial quiver $Q_0$: labels $x_d$ are indicated by $d$, frozen nodes in bold type.}
    \label{FigInitialQuiver}
\end{figure}

Observe that the labels on the nodes of $Q_0$ are precisely the $k(n-k)+1$
variables $x_d$ where $d$ is a rectangular Young diagram, and $n$ of the
nodes are frozen.

\subsection{Mutation step}

Given a quiver with potential $(Q,W)$ as in Definition \ref{DefQuiverPotential},
and fixed a mutable node $v$ of $Q$, one can form a new labeled quiver $Q'$ as follows:

\begin{enumerate}
	\item start with $Q'=Q$, and for all length $2$ paths $a\to v\to b$ with at least one
		mutable node among $a$ and $b$, add to $Q'$ a new edge $a\to b$ ;
	\item modify $Q'$ by reversing all the edges incident to $v$ ;
	\item remove all oriented $2$-cycles formed in $Q'$, by deleting their arrows .
\end{enumerate}

Calling $l(w)$ the label of a node $w$ in $Q$, define new labels $l'(w)$
in $Q'$ by declaring $l'(w)=l(w)$ if $w\neq v$, and
$$l'(v)=\frac{\prod_{w\to v}l(w) + \prod_{v\to w} l(w)}{l(v)} \quad .$$
Since $Q$ and $Q'$ have the same nodes, the nodes of $Q'$ inherit
the property of being frozen or mutable from $Q$.

\begin{definition}\label{DefMutation}
The mutation of $(Q,W)$ along $v$ is the pair $(Q',W')$
with $Q'$ constructed as above, and $W'$ obtained from $W$ by substitution
$l(v)=( \prod_{w\to v}l'(w) + \prod_{v\to w} l'(w) )l'(v)^{-1}$.
\end{definition}

A priori, mutations of quivers with potentials as in Definition \ref{DefQuiverPotential}
are not necessarily quivers with potentials, since $l'(v)$ and
$W'$ are only rational functions of the Pl\"{u}cker coordinates $x_d$.
The following guarantees that certain iterated mutations of the
initial seed of Definition \ref{DefInitialSeed} remain quivers with potentials.

\begin{proposition}(Scott \cite[Theorem 3]{Sc} and Marsh-Rietsch \cite[Section 6.3]{MR})\label{PropPluckerMutations}
Given a finite sequence of mutations that starts at $(Q,W)$ and ends at $(Q',W')$:
\begin{enumerate}
	\item if $(Q,W)=(Q_0,W_0)$ is the initial seed of Definition \ref{DefInitialSeed},
	then $W'$ is a Laurent polynomial in the labels of $Q'$ ;
	\item if in addition each mutation of the sequence is based at some node with
	two incoming and two outgoing edges, then the labels of $Q'$ are Pl\"{u}cker coordinates $x_d$ .
\end{enumerate}
\end{proposition}

\begin{proof}
For the reader's convenience, we explain how the statements follow from the cited results.
It suffices to prove them when the sequence of mutations consists of a single
mutation, as the general case follows by applying repeatedly the same argument.
\begin{enumerate}
\item Marsh-Rietsch \cite[Section 6.3]{MR} (see also Rietsch-Williams \cite[Proposition 9.5]{RW})
showed that the potential $W_0$ of the initial seed is the restriction
$W_0=W_{\vert T_0}$ of a regular function $W\in\mathcal{A}_{k,n}$ to
an algebraic torus $T_0\subset \Gr^\vee(k,n)$ defined by
$$T_0 = \{ [M]\in \Gr^\vee(k,n) \; : \; l(M)\neq 0 \; \forall l \;\textrm{label of}\; Q_0 \} \quad .$$
By Scott \cite[Theorem 3]{Sc} $\mathcal{A}_{k,n}$ is a cluster algebra, and
the rational functions labeling the nodes of $Q'$ are cluster variables.
Just as with the labels of $Q_0$, one can use the labels of $Q'$
to define an algebraic torus $T'\subset\Gr^\vee(k,n)$ via
$$T' = \{ [M]\in \Gr^\vee(k,n) \; : \; l'(M)\neq 0 \; \forall l' \;\textrm{label of}\; Q' \} \quad ;$$
this torus is called toric chart in \cite[Section 6]{Sc}. By Definition \ref{DefMutation},
$W'$ is obtained from $W_0$ by substitution $l(v)=( \prod_{w\to v}l'(w) + \prod_{v\to w} l'(w) )l'(v)^{-1}$.
This means that $W'$ is the pull-back of $W_0$ along the birational map from $T'$ to $T_0$
defined by the substitution formula. It is part of the statement that $\mathcal{A}_{k,n}$
is a cluster algebra that the substitution formula gives a relation
$$l(v)l'(v) - \prod_{w\to v}l'(w) - \prod_{v\to w} l'(w) \quad \in\quad \mathcal{I}_{k,n} \quad ,$$
so that $W'=W_{\vert T'}$ is a restriction of $W$ as well. In particular, $W'$
is a regular function on the algebraic torus $T'$, and hence a Laurent polynomial.
\item If the mutation from $(Q,W)$ to $(Q',W')$ is based at some node $v$ with two
incoming and two outgoing edges, denote $\{v^+_1, v^+_2\}$ and $\{v^-_1,v^-_2\}$
the corresponding nodes of $Q$ adjacent to $v$. The substitution formula of Definition \ref{DefMutation}
simplifies to $l(v)=( l'(v^+_1)l'(v^+_2) + l'(v^-_1)l'(v^-_2))l'(v)^{-1}$. By definition
of mutation $l'(v^+_i)=l(v^+_i)$ and $l'(v^-_i)=l(v^-_i)$ for $i=1,2$. Moreover,
by assumption the $l$ labels are Pl\"{u}cker coordinates, meaning that
$l(v^+_i)=x_{d^+_i}$ and $l(v^-_i)=x_{d^-_i}$ for $i=1,2$ and $l(v)=x_d$ for some Young diagrams
$d, d^+_1, d^+_2, d^-_1, d^-_2\subseteq k\times (n-k)$. Scott \cite[proof of Theorem 3]{Sc}
proves that this implies $l'(v)=x_{d'}$ for some Young diagram $d'\subseteq k\times (n-k)$
too, using combinatorial objects called wiring arrangements. The same phenomenon is
discussed in Rietsch-Williams \cite[Lemma 5.6]{RW} in the combinatorial framework
of plabic graphs; see also the proof of Proposition \ref{PropPluckerDegenerations}
for a comparison between plabic graphs and quivers.
\end{enumerate}
\end{proof}

\begin{definition}\label{DefPluckerSequence}
A length $l$ Pl\"{u}cker sequence of mutations of type $(k,n)$, denoted $\fraks$,
is a finite sequence of pairs $(Q_i,W_i)$ with $0\leq i\leq l$ such that:
\begin{enumerate}
	\item $(Q_0,W_0)$ is the initial seed of type $(k,n)$ of Definition \ref{DefInitialSeed} ;
	\item $(Q_{i+1},W_{i+1})$ is obtained from $(Q_i,W_i)$ by mutation along a mutable
		node with two incoming and outgoing edges, as in Definition \ref{DefMutation} .
\end{enumerate}
If we want to suppress the length $l$, we denote $(Q_l,W_l)=(Q_\fraks, W_\fraks)$
and call it the final quiver with potential of $\fraks$.
\end{definition}

\subsection{Relation with polytope mutations}

Later on in this article, we will be interested in how the Newton polytope $P_\fraks = \Newt(W_\fraks)$
changes throughout a sequence of mutations in the sense of Definition \ref{DefMutation}.

\begin{definition}\label{DefFanoPolytope}
A convex polytope $P\subset\RR^N$ is Fano if the following properties hold:
\begin{enumerate}
	\item $\dim(P)=N$ ;
	\item $0\in\RR^N$ is an interior point of $P$ ;
	\item the vertices $V(P)$ form a set of primitive vectors $V(P)\subset\ZZ^N$ .
\end{enumerate}
\end{definition}

Proposition \ref{PropNewtonFano} proves that each $P_\fraks$ is a Fano polytope.
The Fano condition has the following interpretation in terms of toric geometry;
see \cite{Fu, CLS} for some general background on toric varieties.

\begin{lemma}\label{LemmaToricFano}
If $P$ is a Fano polytope, then the polyhedral fan $\Sigma=\Sigma^fP$ consisting
of the cones spanned by its faces is such that the associated toric variety $X(\Sigma)$
is Fano, meaning that:
\begin{enumerate}
	\item the anti-canonical toric Weil divisor $D_\Sigma$ is $\QQ$-Cartier ;
	\item $D_\Sigma$ is ample .
\end{enumerate}
\end{lemma}

\begin{proof}
The anti-canonical toric Weil divisor is defined to be
$$D_\Sigma = \sum_{\rho\in\Sigma(1)}D_\rho \quad ,$$
where the sum is over the one-dimensional cones $\rho\in\Sigma(1)$ and $D_\rho$ is
the prime divisor corresponding to $\rho$. By \cite[Theorem 4.2.8]{CLS} $D_\Sigma$
is $\QQ$-Cartier if and only if
$$\forall\sigma\in\Sigma(N) \; \exists m_\sigma\in\QQ^N \; : \forall\rho\in\sigma(1) \;\; \langle m_\sigma,u_\rho\rangle = -1 \quad ,$$
where $\Sigma(N)$ is the set of maximal cones of $\Sigma$, $\sigma(1)$ is the set
of one-dimensional cones in $\sigma$, and $u_\rho\in\ZZ^N$ is the primitive generator of $\rho$;
if $m_\sigma$ exists then it is unique, and when all $m_\sigma\in\ZZ^N$ one recorvers
the stronger Cartier condition. By assumption (3) in Definition \ref{DefFanoPolytope}
and the fact that $\Sigma=\Sigma^fP$, one has $u_\rho\in V(P)$ for all $\rho\in\Sigma(1)$.
Now consider the polar dual polytope
$$P^\circ = \{ \; v\in\RR^N \; : \; \langle v,u\rangle \geq -1 \; \forall u\in P \; \} \quad ;$$
since $(P^\circ)^\circ = P$ and $V(P)\subset\ZZ^N$, the polytope $P^\circ$ is reflexive.
One equivalent formulation of reflexivity is to say that each vertex of the polar
dual polytope $u_\rho\in V((P^\circ)^\circ)=V(P)$ defines a facet (or codimension one face)
$$F_{u_\rho} = P^\circ \cap \{ \; v\in\RR^N \; : \; \langle v , u_\rho\rangle = -1 \; \} \quad .$$
We claim that for each $\sigma\in\Sigma(N)$ one has
$$\bigcap_{\rho\in\sigma(1)}F_{u_\rho} = \{m_\sigma\} \quad ,$$
where $m_\sigma\in\QQ^N$ and it satisfies the $\QQ$-Cartier condition. Polar duality
exchanges the face fan $\Sigma^f$ and the normal fan $\Sigma^n$, so that $\sigma\in\Sigma=\Sigma^fP = \Sigma^nP^\circ$;
from this point of view the vector $u_\rho$ can be thought of as inward-pointing
normal to the facet $F_{u_\rho}\subset P^\circ$. By \cite[Proposition 2.3.8]{CLS}
the intersection above describes the unique vertex $m_\sigma\in V(P^\circ)$
corresponding to the $N$-dimensional cone $\sigma$ of the normal fan. Although
in general $m_\sigma\notin\ZZ^N$, one always has $m_\sigma\in\QQ^N$, because
$P^\circ$ is the polar dual of a polytope $P$ with $V(P)\subset\ZZ^N$; see
\cite[Exercise 2.2.1(a)]{CLS}. Finally, $\langle m_\sigma, u_\rho \rangle =-1$
for all $\rho\in\sigma(1)$ because $m_\sigma\in F_{u_\rho}$ by construction.
This proves part (1), for part (2) proceed as follows. The $\QQ$-Cartier divisor
$D_\Sigma$ has a support function $\phi_{D_\Sigma}:\RR^N\to\RR$, which is piecewise-linear
on $\Sigma$ and such that $\phi_{D_\Sigma}(u_\rho)=-1$ for all $\rho\in\Sigma(1)$.
By \cite[Theorem 6.1.7, Lemma 6.1.13]{CLS} $D_\Sigma$ is ample if and only if
the points $\{ \; m_\sigma \; : \sigma\in\Sigma(N) \; \}$ are the vertices of
the polytope
$$P_{D_{\Sigma}}=\{ \; v\in\RR^N \; : \; \langle v,u_\rho\rangle \geq \phi_{D_\Sigma}(u_\rho) \; \forall\rho\in\Sigma(1) \; \} \quad ,$$
and moreover $m_\sigma\neq m_{\sigma'}$ for $\sigma\neq \sigma'$. This is true
because $P_{D_\Sigma}=P^\circ$ and the fact that the correspondence between
faces of $P^\circ$ and cones of its normal fan is a bijection.
\end{proof}

Akhtar-Coates-Galkin-Kasprzyk \cite{ACGK} proposed a general notion of polytope
mutation that should describe how the Newton polytope of a Laurent polynomial changes
under the action of special birational maps of a torus. More precisely, consider
a birational map of the form
$$\phi : (\CC^\times)^N \tto (\CC^\times)^N \quad , \quad \phi = \phi_{M_2} \circ \phi_A \circ \phi_{M_1} \quad ,$$
where
$$ (\CC^\times)^N = \Spec\CC[x_1^{\pm},\ldots ,x_N^{\pm}] \quad \text{and} \quad \phi_A(x_1,\ldots ,x_N)=(x_1,\ldots ,x_{N-1},Ax_N)$$
for some Laurent polynomial $A$ with $\partial_{x_N}A = 0$, and where
$$\phi_{M_i}(x_1,\ldots ,x_N)=(x_1^{m_{11}}\cdots x_N^{m_{1N}},\ldots ,x_1^{m_{N1}}\cdots x_N^{m_{NN}})
	\quad , \quad M_i = (m_{st})_{1\leq s,t\leq N}\in GL(N,\ZZ)$$
for $i=1,2$ are automorphisms of the torus, specified by invertible integer matrices.
If $f$ is a Laurent polynomial, one can think of it as a polynomial in the $x_N^{\pm}$ variables
with coefficients $C_h$ which are Laurent polynomials with $\partial_{x_N}C_h = 0$ and
write
$$f = \sum_{-h_{min}\leq h\leq h_{max}}C_h(x_1,\ldots ,x_{N-1})x_N^h \quad \text{with} \quad h_{min},h_{max}\in\NN \quad .$$
Then the rational function
$$\phi_A^*f = \sum_{-h_{min}\leq h\leq h_{max}} \frac{C_h}{A^{-h}}x_N^h $$
is again a Laurent polynomial whenever $A^{-h} \vert C_t$ for all $-h_{min}\leq h < 0$, and so
is $\phi^*f=g$ because $\phi_{M_1},\phi_{M_2}$ are automorphisms.
The Newton polytopes $P = \Newt(f)$ and $P'=\Newt(g)$ are convex hulls
in $\RR^N$ of the exponent vectors of the monomials in $f$ and $g$ respectively.
The special form of $\phi_A$ singles out the $x_N$ variable, and a width vector
$w =(0,\ldots ,0,1)\in\ZZ^N$ corresponding to this choice. For heights
$-h_{min}\leq  h\leq h$ one can form lattice polytopes $w_h(P)\subseteq P$ by taking
the convex hull of lattice points in hyperplane sections orthogonal to the $w$-direction:
$$w_h(P) = \Conv(P\cap\{\langle \cdot, w\rangle = h \}\cap\ZZ^N ) \quad .$$
In fact, $h_{max} - h_{min}\in\NN$ can be thought as the width of the polytope $P$
with respect to the $w$-direction, and $h$ as a height coordinate. Calling
$$ F = \Newt(A) \quad \text{and} \quad G_h = \Newt\left(\frac{C_h}{A^{-h}}\right) \quad \text{for} \; -h_{min}\leq h\leq h_{max} \quad ,$$
the polytope $F\subset\RR^N$ has codimension at least one and lies at height $h=0$.
Denoting $V(P)$ the vertices of $P$ one has
$$V(P)\cap\{\langle \cdot, w\rangle = h \} \subseteq G_h + (-h)F \subseteq w_h(P) \quad \forall \, -h_{min}\leq h < 0 \quad .$$
The notation $P' = \mut_w(P,F)$ expresses the fact that $P'$ is a polytope mutation
of $P$ in direction $w$ and with factor $F$, and the quantity $h_{max} - h_{min}$
is called width of the mutation.
We now explain how mutation in the sense of Definition \ref{DefMutation} is related
to polytope mutations.

\begin{proposition}\label{PropNewtonFano}
If $(Q',W')$ is obtained from $(Q,W)$ by mutation along a node $v$, then the Newton
polytopes $P=\Newt(W)$ and $P=\Newt(W')$ are related by polytope mutation. In particular,
$P_\fraks$ and $X(\Sigma^fP_\fraks)$ are Fano for any Pl\"{u}cker sequence $\fraks$.
\end{proposition}

\begin{proof}
Define the complex tori
$$T=\Spec(\CC[x_d^\pm : d \; \textrm{label of} \; Q])\quad , \quad T'=\Spec(\CC[x_d^\pm : d \; \textrm{label of} \; Q']) \quad ,$$
and think of $W$ and $W'$ as regular functions on them.
Quiver mutation along $v$ changes the label $x_{l(v)}$ of $Q$ into a new label $x_{l'(v)}$
in $Q'$, and the two are related by
$$x_{l(v)}x_{l'(v)} = \prod_{w\to v}x_{l(w)} + \prod_{v\to w} x_{l(w)} \quad .$$
Up to automorphisms of tori, one can arrange the coordinates
in such a way that $x_{l(v)}$ and $x_{l'(v)}$ go last, and the common ones appear in the same order.
With this choice, there is a birational transition map between the tori
$$\phi = (\operatorname{id}, (\prod_{w\to v}x_{l(w)} + \prod_{v\to w} x_{l(w)})x_{l(v)}^{-1}) \quad .$$
The Laurent polynomial $A=\prod_{w\to v}x_{l(w)} + \prod_{v\to w} x_{l(w)}$ satisfies $\partial_{x_{l(v)}}A=0$,
and using the notation introduced in this section $\phi = \phi_A\circ\phi_{M_1}$,
where $\phi_{M_1}$ is the automorphism of $T$ that inverts the last coordinate.
By direct inspection, one sees that the polytope $P_0=\Newt(W_0)$ corresponding to
the initial potential $W_0$ given in Definition \ref{DefInitialSeed} is Fano in
the sense of Definition \ref{DefFanoPolytope}. It follows from \cite[Proposition 2]{ACGK} that $P_\fraks$ is Fano
for every Pl\"{u}cker sequence $\fraks$, and thus $X(\Sigma^fP_\fraks)$ is too,
thanks to Lemma \ref{LemmaToricFano}.
\end{proof}

\section{Pl\"{u}cker Lagrangians}\label{SecLagrangians}

In this section, $\Sigma$ denotes a complete fan in $\RR^{k(n-k)}$, and $X(\Sigma)$
its associated proper toric variety; see for example \cite{Fu, CLS} for background
material on toric geometry. The reader familiar with symplectic manifolds and
Hamiltonian torus actions can think of $\Sigma$ as the normal fan $\Sigma = \Sigma^n\Delta$
of a moment polytope $\Delta$, with the important caveat that $X(\Sigma)$ is
typically singular, and not even an orbifold; in this case $\Delta$ should be
thought as the closure of the open convex region obtained from the moment
map of the maximal torus orbit.
\par
We will assume that the primitive generators of the rays of $\Sigma$ in the
lattice $\ZZ^{k(n-k)}\subset\RR^{k(n-k)}$ are the vertices of a convex polytope
$P$, and alternatively think of $\Sigma$ as its face fan $\Sigma = \Sigma^fP$.
This condition is equivalent to $X(\Sigma)$ being Fano, and $P$ is sometimes
called a Fano polytope. The reader should not confuse the polytopes $\Delta$
and $P$: the second is always a lattice polytope, whereas the first may not be.
The two polytopes are related by polar duality $\Delta = P^\circ$.

\subsection{Lagrangian tori from degenerations}

\begin{definition}\label{DefToricDegeneration}
If $X\subset\PP^M$ is a smooth subvariety of complex dimension $N$, an
embedded toric degeneration
$X\rightsquigarrow X(\Sigma)$ is a closed subscheme $\cX\subset\PP^M\times\CC$
such that the map $p:\cX\to\CC$ obtained by restriction of the projection
satisfies the following properties:
\begin{itemize}
	\item $p^{-1}(\CC^\times) \cong X\times\CC^\times$ as schemes over $\CC^\times$ ;
	\item $p^{-1}(0)\subset\PP^M$ is an orbit closure for some linear torus action $(\CC^\times)^N\curvearrowright\PP^M$ ;
	\item $p^{-1}(0)$ is a toric variety with fan $\Sigma$ .
\end{itemize}
\end{definition}

\begin{proposition}\label{PropPluckerDegenerations}(Rietsch-Williams \cite[Theorem 1.1]{RW})
Every Pl\"{u}cker sequence $\fraks$ of mutations of type $(k,n)$
has an associated embedded toric degeneration $\Gr(k,n)\rightsquigarrow X(\Sigma_\fraks)$,
where $\Sigma_\fraks = \Sigma^fP_\fraks$ is the face fan of the Newton polytope
$P_\fraks$ of the final potential $W_\fraks$.
\end{proposition}

\begin{proof}
For the reader's convenience, we provide details on how to specialize the result
of Rietsch-Williams \cite[Theorem 1.1]{RW} to recover this statement.
Each step $(Q_i,W_i)$ of the Pl\"{u}cker sequence $\fraks$ corresponds to a
reduced plabic graph $G_i$ of type $\pi_{k,n}$ \cite[Section 3]{RW}, which is
a combinatorial object encoding the quiver $Q_i$ and the Laurent polynomial
$W_i$ simultaneously. Nodes in $Q_i$ correspond to
faces in $G_i$, and each arrow of $Q_i$ is dual to an edge of $G_i$, with
black/white nodes of the plabic graph respectively to the right/left of the arrow.
The frozen nodes of $Q_i$ correspond to boundary faces of $G_i$, and the mutable
nodes to interior faces. Mutations at some mutable node with two incoming and outgoing
arrows in $Q_i$ correspond to a square move on the plabic graph $G_i$. The
Pl\"{u}cker variables on nodes of $Q_i$ are labeled by the Young diagrams
appearing on the faces of $G_i$, which are induced by trips as in \cite[Definition 3.5]{RW}.
The Laurent polynomial $W_i$ is a generating function
counting matchings on the plabic graph $G_i$ \cite[Theorem 18.2]{RW}; see also
Marsh-Scott \cite{MS} for a proof. The initial seed $(Q_0,W_0)$ corresponds
to a particular plabic graph $G_0=G^{rec}_{k,n}$, called the rectangle plabic
graph in \cite[Section 4]{RW}. Consider
the divisor $D_i\subset\Gr(k,n)$ cut out by the equation $x_{d_i}=0$, with
$d_i\subseteq k\times (n-k)$ one of the $n$ frozen Young diagrams, and call
$D=r_1D_1+\cdots +r_nD_n$ a general effective divisor with the same support.
One can associate to the pair $(D,G_\fraks)$ a convex polytope $\Delta_{G_\fraks}(D)$
known as Okounkov body \cite[Section 1.2]{RW}. From now on set
$r_1=\ldots =r_n=1$, and call $D_{FZ}=D_1+\cdots +D_n$ the corresponding divisor.
There exists a scaling factor $r_\fraks\in\QQ^+$
such that $r_\fraks\Delta_{G_\fraks}(D_{FZ})$ is a normal lattice polytope \cite[Definition 2.2.9]{CLS};
normality is referred to as integer decomposition
property in \cite[Definition 17.7]{RW}, and from \cite[Proposition 19.4]{RW}
one sees that the scaling factor mentioned there is related to ours by $r_\fraks=\frac{r_{G_\fraks}}{n}$.
From \cite[Section 17]{RW} one gets a degeneration of $\Gr(k,n)$ to the toric variety
associated with the polytope $r_\fraks\Delta_{G_\fraks}(D_{FZ})$, and this
is an embedded toric degeneration in the sense of Definition \ref{DefToricDegeneration}
with fan $\Sigma_\fraks=\Sigma^nr_\fraks\Delta_{G_\fraks}(D_{FZ})=\Sigma^n\Delta_{G_\fraks}(D_{FZ})$,
where we used that the normal fan of a polytope doesn't change under scaling. In \cite[Theorem 1.1]{RW} and
\cite[Definition 10.14]{RW}, an interpretation of $\Delta_{G_\fraks}(r_1D_1+\cdots+r_nD_n)$ is given
in terms of the tropicalization of $W_\fraks$. Setting $r_1=\ldots =r_n=1$, one finds in particular that for $D_{FZ}=D_1+\cdots + D_n$ in fact
$$\Delta_{G_\fraks}(D_{FZ})=\{ v\in\RR^{k(n-k)} \; : \; \langle v,u\rangle \geq -1 \;
\textrm{for every vertex} \; u\in P_\fraks\} \quad ;$$
here $P_\fraks$ denotes the Newton polytope of the Laurent polynomial $W_\fraks$,
i.e. the convex hull of its exponents. To see this, observe that from \cite[Theorem 1.1]{RW}
and \cite[Definitions 10.7, 10.14]{RW} one has
$$v\in\Delta_{G_\fraks}(D_{FZ}) \iff \Trop(W_{i|T_\fraks})(v)\geq -1 \; \textrm{for} \; i=1,\ldots ,n \quad ;$$
here each $W_i$ is a special term of  a rational function $W=W_1+\ldots +W_n$ on $\Gr^\vee(k,n)$
defined in \cite[Definition 10.1]{RW}, and
$$T_\fraks = \{ \; [M]\in\Gr^\vee(k,n) \; : \; l(M)\neq 0 \; \forall l \textrm{ label of } Q_\fraks \; \}$$
is a complex torus chart such that $W_{|T_\fraks}=W_\fraks$; compare (1) of Proposition \ref{PropPluckerMutations}.
The symbol $\Trop(\cdot)$ denotes tropicalization of Laurent polynomials, which
produces a piece-wise linear function defined as
$$\Trop\left(\sum_u c_u x^u \right)(v) = \operatorname{min}_{u}\langle v, u\rangle \quad .$$
Also observe that
$$\Trop(W_{i|T_\fraks})(v)\geq -1 \; \textrm{for} \; i=1,\ldots n \quad
\iff \operatorname{min}_{i=1,\ldots ,n}\Trop(W_{i|T_\fraks})(v)\geq -1$$
$$\iff \Trop(W_{|T_\fraks})(v) \geq -1 \iff \Trop(W_\fraks)(v) \geq -1 \quad .$$
Summarizing, $v\in\Delta_{G_\fraks}(D_{FZ})$ is equivalent to
$\langle v, u\rangle \geq -1$ for every $u$ exponent of a monomial in $W_\fraks$.
By convexity, the latter condition is equivalent to asking $\langle v, u\rangle \geq -1$ only for
those $u$ that are vertices of the Newton polytope $P_\fraks$ of $W_\fraks$.
We have thus recovered the polar dual polytope, i.e. $\Delta_{G_\fraks}(D_{FZ})=P_\fraks^\circ$. Since the
normal fan of a polytope equals the face fan of its polar dual and polar duality
is an involution, we find that
$\Sigma_\fraks = \Sigma^n\Delta_{G_\fraks}(D_{FZ}) = \Sigma^f P_\fraks$ as in the statement.

\end{proof}

\begin{remark}
The toric variety $X(\Sigma_\fraks)$ depends only the final step of $\fraks$ in
the following sense. Suppose that $Q_\fraks$ and $Q_{\fraks'}$ are the final quivers of
two different Pl\"{u}cker sequences. Consider the charts
$$T_\fraks = \{ \; [M]\in\Gr^\vee(k,n) \; : \; l(M)\neq 0 \; \forall l \textrm{ label of } Q_\fraks \; \} \quad \textrm{and}$$
$$T_{\fraks'} = \{ \; [M]\in\Gr^\vee(k,n) \; : \; l'(M)\neq 0 \; \forall l' \textrm{ label of } Q_{\fraks'} \; \} \quad ;$$
these were already considered in (1) of Proposition \ref{PropPluckerMutations},
were it was observed that $W_\fraks = W_{|T_\fraks}$ and $W_{\fraks'} = W_{|T_\fraks}$.
If the quivers $Q_\fraks$ and $Q_\fraks'$ have equal sets of labels, then
$T_\fraks = T_{\fraks'}$ and therefore the final Laurent polynomials of the two
Pl\"{u}cker sequences are equal: $W_\fraks = W_{\fraks'}$. Since $\Sigma_\fraks$
and $\Sigma_{\fraks'}$ are the face fans of their Newton polytopes $P_\fraks=P_{\fraks'}$,
it follows that $\Sigma_{\fraks}=\Sigma_{\fraks'}$ and thus $X(\Sigma_{\fraks})=X(\Sigma_{\fraks'})$.
\end{remark}

In what follows, we endow the Grassmannian $\Gr(k,n)$ with the symplectic
structure obtained by restriction of the Fubini-Study form on the target
projective space of the Pl\"{u}cker embedding.

\begin{proposition}\label{PropPluckerTori}(Harada-Kaveh \cite[Theorem B]{HK})
Every Pl\"{u}cker sequence $\fraks$ of mutations of type $(k,n)$
has an associated Lagrangian torus $L_\fraks\subset\Gr(k,n)$, and it comes
with a canonical basis of $H_1(L_\fraks;\ZZ)$.
\end{proposition}

\begin{proof}
For the reader's convenience, we provide details on how to specialize the result
of Harada-Kaveh \cite[Theorem B]{HK} to recover this statement. Recall from
Proposition \ref{PropPluckerDegenerations} that $\fraks$ determines a
degeneration of $\Gr(k,n)$ to the toric variety
associated with the polytope $r_\fraks\Delta_{G_\fraks}(D_{FZ})$, known as Okounkov body;
for a detailed description of the value semigroup underlying this Okounkov body
and of why it satisfies the assumptions of \cite[Theorem B]{HK}, see \cite[Definition 17.8, Lemma 17.9]{RW}.
From \cite[Theorem B]{HK} one deduces the existence of an open set $U_\fraks\subset\Gr(k,n)$
and a smooth submersion $\mu_\fraks : U_\fraks\to \RR^{k(n-k)}$ whose image
is the interior of the polytope, and whose fibers are Lagrangian tori. Call
$L_\fraks=\mu_\fraks^{-1}(0)$. If $p\in L_\fraks$, the tangent space to the fiber
at $p$ is $T_p(L_\fraks)=\operatorname{ker}(d_p\mu_\fraks)$. Therefore, the standard
basis of $\RR^{k(n-k)}$ lifts under $d_p\mu_\fraks$ to a basis of $T_p(U_\fraks)/T_p(L_\fraks)$. Since
the symplectic structure vanishes on $L_\fraks$, the lift defines
a symplectic-dual basis of $T_p(L_\fraks)$. Since $L_\fraks$ is a torus, the vectors of
this basis are tangent to natural closed loops in $L_\fraks$, and their homology classes give
a basis of $H_1(L_\fraks;\ZZ)$ which is independent of the point $p\in L_\fraks$.
\end{proof}

\begin{definition}\label{DefPluckerTorus}
The Lagrangians $L_\fraks\subset\Gr(k,n)$ of Proposition \ref{PropPluckerTori}
are called Pl\"{u}cker tori, and the elements $\gamma_d\in H_1(L_\fraks;\ZZ)$ of
the canonical basis are called canonical cycles. We index the canonical cycles
by Young diagrams $d\subseteq k\times(n-k)$ such that $d\neq\emptyset$
and $d$ appears in a label $x_d$ of the final quiver $Q_\fraks$.
\end{definition}

\begin{remark}
Note that $x_\emptyset$ is the label of a frozen node in the initial quiver $Q_0$.
Since frozen labels do not change under mutation, $x_\emptyset$ is in fact a the
label of a forzen node in any quiver $Q_\fraks$.
\end{remark}

\subsection{Local systems from Schur polynomials}

Given a Young diagram $d\subseteq k\times (n-k)$, the corresponding $k$-variables
Schur polynomial is defined as
$$S_d(z_1,\ldots z_k)=\sum_{T_d}z_1^{t_1}\cdots z_k^{t_k} \quad ,$$
where the sum runs over semi-standard tableaux $T_d$ on $d$. The tableaux $T_d$ are obtained by labeling $d$
with integers in $\{1,\ldots k\}$, in such a way that rows are weakly increasing
and columns are strictly increasing. The exponent $t_i$ is the number of times
that the integer $i$ appears in the tableaux $T_d$.
If $I$ is any of the ${n \choose k}$ sets of roots of $z^n=(-1)^{k+1}$
with $|I|=k$, it makes sense to evaluate $S_d(I)\in\CC$ without specifying an
order on the elements of $I$ because Schur polynomials are symmetric functions.

\begin{definition}\label{DefLocalSystems}
For each Pl\"{u}cker torus $L_\fraks\subset\Gr(k,n)$ and set $I$, denote $\xi_I$ the
rank one local system whose holonomy $\hol_{\xi_I}:H_1(L_\fraks;\ZZ)\to\CC^\times$
around the canonical cycles of Definition \ref{DefPluckerTorus} is given by the formula
$$\hol_{\xi_I}(\gamma_d)=S_d(I) \in \CC^\times \quad ;$$
if $S_d(I)=0$ for some $d$ appearing in a label $x_d$ of the final quiver of $\fraks$,
then $\xi_I$ is not defined.
\end{definition}

\subsection{A conjecture and some evidence}

If $\fraks$ is a Pl\"{u}cker sequence of type $(k,n)$, after setting
$x_\emptyset = 1$ the Laurent polynomial $W_\fraks$ can be thought of as a regular
function on the algebraic torus $H_1(L_\fraks ; \ZZ)\otimes\CC^\times\cong (\CC^\times)^{k(n-k)}$.
Setting $x_\emptyset = 1$ corresponds to thinking $\mathcal{A}_{k,n}=\mathcal{O}(U_{k,n})$
as algebra of regular functions on $U_{k,n}=\Gr^\vee(k,n)\setminus D^\vee_{FZ}$
rather than on its affine cone.

The canonical cycles $\gamma_d\in H_1(L_\fraks;\ZZ)$ of Definition \ref{DefPluckerTorus}
give an isomorphism of schemes $H_1(L_\fraks;\ZZ)\otimes\CC^\times\cong (\CC^\times)^{k(n-k)}$,
where one thinks the latter torus as having coordinates $x_d$ labeled by
Young diagrams $d\subseteq k\times (n-k)$ such that $d\neq\emptyset$ and $d$
appears in some label of the quiver $Q_\fraks$.
Under the identification described by Scott \cite[Theorem 4]{Sc}, one can think
that $H_1(L_\fraks;\ZZ)\otimes\CC^\times\cong T_\fraks\subset\Gr^\vee(k,n)$ where
$$T_\fraks = \{ [M]\in\Gr^\vee(k,n) \; : \; x_d(M)\neq 0 \; \forall d \; \textrm{label of} \; Q_\fraks \} \quad .$$

\begin{conjecture}\label{ConjPotentials}
If $\fraks$ and $\fraks'$ are two Pl\"{u}cker sequences of type $(k,n)$, and
$\phi$ is a Hamiltonian isotopy of $\Gr(k,n)$ such that $\phi(L_\fraks)=L_{\fraks'}$,
then the induced map $\phi_*: H_1(L_\fraks;\ZZ)\to H_1(L_{\fraks'};\ZZ)$ is such that
$$W_{\fraks} \sim W_{\fraks'} \circ (\phi_* \otimes \operatorname{id}_{\CC^\times}) \quad ,$$
where $\sim$ denotes equality up to automorphisms of $T_\fraks$.
\end{conjecture}

\begin{remark}
The reason for $\sim$ in the conjecture above is the following. Suppose that
$W_\fraks$ is the disk potential of $L_\fraks$, i.e. $W_\fraks = W_{L_\fraks}$.
By Hamiltonian invariance of the disk potential, if $L_{\fraks'}=\phi(L_\fraks)$
then $W_{L_{\fraks'}}=W_{L_\fraks}=W_\fraks$ as long as we express the disk
potential of $L_{\fraks'}$ in the basis of cycles induced by $\phi_*: H_1(L_\fraks;\ZZ)\to H_1(L_{\fraks'};\ZZ)$.
Instead, the Laurent polynomial $W_{\fraks'}$ expresses the disk potential
of $L_{\fraks'}$ in the canonical basis of cycles of Definition \ref{DefPluckerTorus},
which is a priori different from the one induced by $\phi_*$.
\end{remark}

Under some assumptions on the singularities of the toric varieties $X(\Sigma_\fraks)$
appearing as limits of the degenerations $\Gr(k,n)\rightsquigarrow X(\Sigma_\fraks)$,
the conjecture above can be verified. We describe below how, and give some sample
applications in Section \ref{SecApplications}.

\begin{definition}\label{DefToricResolution}
If $X(\Sigma)$ is a projective toric variety, a toric resolution consists of a
smooth projective toric variety $X(\widetilde{\Sigma})$ with a toric morphism
$r:X(\widetilde{\Sigma})\to X(\Sigma)$ which is a birational equivalence.
\end{definition}

Any toric variety $X(\Sigma)$ has a toric resolution;
see for example \cite[Chapter 11]{CLS}. Toric resolutions can be constructed by
taking refinements $\widetilde{\Sigma}$ of the fan $\Sigma$, which have natural
associated morphisms $r$.
The refined fan $\widetilde{\Sigma}$ has in general more rays than $\Sigma$, and these
correspond to torus invariant divisors in the exceptional locus $r^{-1}(\Sng X(\Sigma))$.

\begin{definition}\label{DefSmallResolution}
A toric resolution $r:X(\widetilde{\Sigma})\to X(\Sigma)$ is small if
$\widetilde{\Sigma}$ and $\Sigma$ have the same rays. Being small is equivalent to
$\operatorname{codim}_{\CC}(r^{-1}(\Sng X(\Sigma)))\geq 2$; see for example \cite[Proposition 11.1.10]{CLS}.
\end{definition}

\begin{proposition}\label{PropSmallMonotone}
If $\fraks$ is a Pl\"{u}cker sequence of type $(k,n)$, and the toric variety $X(\Sigma_\fraks)$
admits a small resolution, then $L_\fraks\subset\Gr(k,n)$ is monotone and has disk
potential $W_\fraks$ with respect to the basis of canonical cycles for $H_1(L_\fraks; \ZZ)$.
\end{proposition}

\begin{proof}
Recall from Proposition \ref{PropPluckerTori} that there is a smooth submersion
$\mu_\fraks : U_\fraks \to \RR^{k(n-k)}$ with Lagrangian
torus fibers, defined on some open set $U_\fraks\subset\Gr(k,n)$. If $P_\fraks$ is
the Newton polytope of the Laurent polynomial $W_\fraks$, the image of this map
is the interior of the polytope described in Proposition \ref{PropPluckerDegenerations}:
$$r_\fraks\Delta_{G_\fraks}(D_{FZ})=\{ v\in\RR^{k(n-k)} \; : \; \langle v,u\rangle \geq -r_\fraks \;
\textrm{for every vertex} \; u\in P_\fraks\} \quad .$$
Call $v$ a point in the interior, and $L_\fraks(v)=\mu_\fraks^{-1}(v)$ the corresponding
Lagrangian torus fiber. Observe that $X(\Sigma_\fraks)$ is Fano, thanks to
Proposition \ref{PropNewtonFano}. The assumption that $X(\Sigma_\fraks)$ has a small toric
resolution allows to invoke a theorem of Nishinou-Nohara-Ueda \cite[Theorem 10.1]{NNU},
and conclude that the disk potential of $L_\fraks(v)\subset\Gr(k,n)$ has one monomial
for each facet $\langle v,u\rangle = -r_\fraks$ of $r_\fraks\Delta_{G_\fraks}(D_{FZ})$,
with exponent $u\in H_1(L_\fraks(v);\ZZ)$.
The subgroup of disk classes in $H_2(\Gr(k,n),L_\fraks(v);\ZZ)$ is
generated by Maslov index 2 classes. This holds because $U_\fraks\subset\Gr(k,n)$
is the domain of a symplectomporhism $\psi : U_\fraks \to X(\Sigma_\fraks)\setminus D_{\Sigma_\fraks}$
with the maximal torus orbit of the toric variety $X(\Sigma_\fraks)$ obtained
by removing the standard torus invariant divisor $D_{\Sigma_\fraks}$. Harada-Kaveh \cite[Theorem A (1)]{HK}
shows that $\psi$ extends to a continuous map $\overline{\psi}:\Gr(k,n)\to X(\Sigma_\fraks)$.
As explained in Nishinou-Nohara-Ueda \cite[Lemma 9.2, Corollary 9.3]{NNU}, the
assumption that $X(\Sigma_\fraks)$ has a small resolution $X(\widetilde{\Sigma_\fraks})$
allows to use the map $\overline{\psi}$ to identify disk classes in $H_2(\Gr(k,n),L_\fraks(v);\ZZ)$ with
classes of disks in $X(\widetilde{\Sigma_\fraks})$ with boundary on the toric moment fiber over $v$,
and these are generated by Maslov index 2 classes; see Cho-Oh \cite[Theorem 5.1]{CO}
for a general formula computing the Maslov index of disks with boundary in a toric
moment fiber.
In conclusion, $L_\fraks(v)$ is monotone if and only if all Maslov index 2 classes
have the same symplectic area. The symplectic area of a Maslov 2 disk with boundary $u$ is 
$2\pi(\langle v,u\rangle + r_\fraks) ;$
see Cho-Oh \cite[Theorem 8.1]{CO} for a proof. The choice $v=0$ guarantees that
all the areas are equal, and thus $L_\fraks=L_\fraks(0)$ is monotone.
\end{proof}

\section{Sample applications}\label{SecApplications}

We describe some sample applications of what seen so far to the symplectic topology
of Grassmannians. These results are by no means optimal; they are meant
to illustrate new phenomena, and we give some indications on how one can prove
analogous statements using the same techniques.

\subsection{Generating the Fukaya category of $\Gr(2,n)$}

In this subsection we focus on Grassmannians of planes $\Gr(2,n)$. For this
class of Grassmannians, Nohara-Ueda \cite{NU14, NU20} have used symplectic
reduction techniques to construct a Catalan number $C_{n-2}$ of integrable systems
on $\Gr_2(n)$, each labeled by a triangulation $\Gamma$ of the $n$-gon; the generic fibers
of these systems are Lagrangian tori $L_\Gamma\subset\Gr_2(n)$, and the images
of their Hamiltonians are lattice polytopes $\Delta_\Gamma$. Explicit formulas
for the disk potentials $W_{L_\Gamma}$ were given as sums over edges of the
triangulation $\Gamma$. We compare this construction with the case $k=2$ of
our Construction \ref{ConTori}, and explore some consequences for the Fukaya
category of $\Gr(2,n)$.
When $k=2$, the Pl\"{u}cker coordinates appearing as labels of a quiver $Q_\fraks$ have a
simple combinatorial description.

\begin{definition}\label{DefTriangulation}
Let $n>2$, and consider $n$ points on $S^1$, labeled counter-clockwise from $1$ to $n$.
A triangulation $\Gamma$ of $[n]$ is a collection of subsets $\{i,j\}\subset[n]$
with $i\neq j$, such that connecting $i$ and $j$ with an arc in $D^2$ for all
$\{i,j\}\in\Gamma$ one gets a triangulation of the $n$-gon with vertices at $[n]$.
\end{definition}

\begin{remark}
Every triangulation $\Gamma$ of $[n]$ must contain the $n$ sets
$$\{1,2\} \; , \; \{2,3\} \; , \; \ldots \; , \;\{n-1,n\} \; , \; \{1,n\} \quad ;$$
these correspond to the edges of the $n$-gon; the other sets in $\Gamma$
correspond to interior edges of the triangulation. Note that the $n$ sets
above are also the vertical steps $d^|$ of those Young diagrams $d\subseteq 2\times (n-2)$
that label the frozen nodes in Definition \ref{DefInitialSeed} (specialized to $k=2$).
\end{remark}

\begin{lemma}(Fomin-Zelevinsky \cite[Proposition 12.5]{FZ2}, \cite[Theorem 1.6]{OPS})\label{LemmaTriangClusters}
A collection of Young diagrams $d\subseteq 2\times(n-2)$ labels
the nodes of $Q_\fraks$ for some Pl\"{u}cker sequence $\fraks$ of type $(2,n)$
if and only if the set $\Gamma = \{ d^| \subset [n] \}$ is a triangulation of $[n]$.
\end{lemma}

\begin{proof}
A set of Pl\"{u}cker coordinates $x_d$ labels the nodes of some quiver $Q_\fraks$ precisely
if it is a cluster in the cluster algebra structure of $\mathcal{A}_{2,n}$. By Oh-Postnikov-Speyer
\cite[Theorem 1.6]{OPS}, this is equivalent to saying that the sets $d^|$ of vertical steps of the corresponding
Young diagrams $d\subseteq 2\times (n-2)$ form a maximal weakly separated collection
in the sense of \cite[Definition 3.1]{OPS}. For $k=2$ the general notion of
maximal weakly separated collection recovers the one of triangulation given in
Definition \ref{DefTriangulation}.
\end{proof}

\begin{lemma}(Nohara-Ueda \cite[Theorem 1.5]{NU14} and \cite[Theorem 1.1]{NU20})\label{LemmaK2}
If $k=2$ then the Lagrangian tori $L_\fraks\subset \Gr(2,n)$ are monotone for all $\fraks$ and have disk
potential $W_\fraks$. Consequently, by Hamiltonian invariance of the disk potential,
Conjecture \ref{ConjPotentials} holds.
\end{lemma}

\begin{proof}
In view of Proposition \ref{PropSmallMonotone}, it suffices to prove that
for any Pl\"{u}cker sequence $\fraks$ of type $(2,n)$ the
toric variety $X(\Sigma_\fraks)$ has a small toric resolution. From Lemma
\ref{LemmaTriangClusters}, the labels of $Q_\fraks$ define a triangulation
$\Gamma_\fraks$ of $[n]$. Nohara-Ueda \cite[Theorem 1.1]{NU20} describe
an open embedding $\iota_{\Gamma_\fraks}:(\CC^\times)^{2(n-2)}\to\Gr^\vee(k,n)$
such that $\iota_{\Gamma_\fraks}^*W=W_{\Gamma_\fraks}$, where $W\in\mathcal{A}_{k,n}$ is the
Landau-Ginzburg potential defined by Marsh-Rietsch \cite{MR} and $W_{\Gamma_\fraks}$ is
a Laurent polynomial associated to the triangulation $\Gamma_\fraks$. It was
shown earlier by Nohara-Ueda \cite[Proposition 7.4]{NU14} that the polar dual
of the Newton polytope of $W_{\Gamma_\fraks}$ is a lattice polytope
$\Delta_{\Gamma_\fraks}=\Newt^\circ(W_{\Gamma_\fraks})$ (as opposed to just
a rational polytope) and that the associated toric variety $X(\Sigma^n\Delta_{\Gamma_s})$
has a small toric resolution \cite[Theorem 1.5]{NU14}. Since polar duality
exchanges normal and face fans $X(\Sigma^n\Delta_{\Gamma_\fraks})=X(\Sigma^f\Newt(W_{\Gamma_\fraks}))$.
The image of the embedding $\iota_{\Gamma_\fraks}$ is the cluster chart
$T_\fraks$ and $W_\fraks=W_{|T_\fraks}$, so $W_{\Gamma_\fraks}$ and $W_\fraks$ are
Laurent polynomials related by an
automorphism of the torus, therefore their Newton polytopes $\Newt(W_{\Gamma_\fraks})$
and $P_\fraks$ are equivalent under the action of $GL(2(n-2),\ZZ)$. We conclude
that $X(\Sigma^f\Newt(W_{\Gamma_\fraks}))\cong X(\Sigma^fP_\fraks)=X(\Sigma_\fraks)$
and therefore $X(\Sigma_\fraks)$ has a small toric resolution too.
\end{proof}

As described by Sheridan \cite{Sh}, the Fukaya category of a monotone symplectic
manifold like the Grassmannian has a spectral decomposition 
$$\Fuk(\Gr(k,n))=\bigoplus_{\lambda}\Fuk_\lambda(\Gr(k,n)) \quad .$$
The summands are $A_\infty$-categories indexed by the eigenvalues $\lambda$ of
the operator $c_1\star$ of multiplication by the first Chern class acting on the small
quantum cohomology. The objects of the $\lambda$-summand
are monotone Lagrangians with rank one local systems $L_\xi$ as described in Section \ref{SecIntro}.
The following proposition holds for general Grassmannians.

\begin{proposition}\label{PropCritPointsValues}
For any $1\leq k<n$ and any Pl\"{u}cker sequence $\fraks$ of type $(k,n)$:
\begin{enumerate}
\item if $\Fuk_\lambda(\Gr(k,n))\neq 0$ then $\lambda = n(\zeta_1+\ldots + \zeta_k)$
for some $\{\zeta_1, \ldots ,\zeta_k\}=I\subset\{\;\zeta\in\CC : \zeta^n=(-1)^{k+1}\;\}$
with $|I|=k$ ;
\item if $X(\Sigma_\fraks)$ has a small toric resolution, then
$(L_\fraks)_{\xi_I}$ is a defined and nonzero in $\Fuk_{\lambda}(\Gr(k,n))$ if and only if $S_d(I)\neq 0$
for all Young diagrams $d$ appearing as labels on the nodes of $Q_\fraks$, and
moreover $\lambda=n(\zeta_1+\cdots +\zeta_k)$ .
\end{enumerate}
\end{proposition}

\begin{proof}
\begin{enumerate}
\item By \cite[Proposition 1.3]{Ca}, $\lambda\in\CC$ is an eigenvalue
of the operator $c_1\star$ of multiplication by the first Chern class acting on the small
quantum cohomology if and only if $\lambda = n(\zeta_1+\ldots + \zeta_k)$
for some $\{\zeta_1, \ldots ,\zeta_k\}=I\subset\{\;\zeta\in\CC : \zeta^n=(-1)^{k+1}\;\}$
with $|I|=k$. By Auroux \cite[Proposition 6.8]{Au}, any monotone
Lagrangian $L$ with a rank one local system $\xi$ having Floer cohomology
$HF(L_\xi,L_\xi)\neq 0$ must have $m^0(L_\xi)$ which is an eigenvalue of $c_1\star$.
\item Think of the holonomy of a local system $\xi$ on $L_\fraks$ as a point on a complex torus
$$\operatorname{hol}_\xi\in\operatorname{Hom}(H_1(L_\fraks;\mathbb{Z}))\cong(\CC^\times)^{k(n-k)}\quad .$$
The identification depends on the choice of a basis for $H_1(L_\fraks;\ZZ)$, and
we use the canonical $\gamma_d\in H_1(L_\fraks;\ZZ)$ of Definition \ref{DefPluckerTorus}.
By Proposition \ref{PropSmallMonotone} and the assumption of small resolution,
the Lagrangian torus $L_\fraks$ is monotone and has disk potential $W_{L_\fraks}=W_\fraks$.
By Auroux \cite[Proposition 6.9]{Au} and Sheridan \cite[Proposition 4.2]{Sh},
one has Floer cohomology $HF((L_\fraks)_\xi,(L_\fraks)_\xi)\neq 0$ if and only if
$\hol_\xi$ is a critical point of the disk potential. By Definition \ref{DefLocalSystems},
the local system $\xi=\xi_I$ has $\hol_{\xi_I}(\gamma_d)=S_d(I)$. Rietsch \cite[Lemma 4.4]{Ri}
proves that $S_d(I)=S_{d^T}(I^\vee)$, where $d^T\subseteq (n-k)\times k$ is the transpose Young diagram
of $d$ and $I^\vee=\{\zeta_1,\ldots \zeta_{n-k}\}$ is the set of $n-k$ distinct roots
of $\zeta^n = (-1)^{n-k+1}$ obtained by looking at the roots $I^c$ of $\zeta^n = (-1)^{k+1}$
that are not in $I$ and declaring $I^\vee=e^{\pi i}I^c$. Consider the points
$[M_{I^\vee}]\in\Gr^\vee(k,n)$ defined as
$$[M_{I^\vee}]=
\begin{bmatrix}
    1 & 1 & 1 & \dots  & 1 \\
    \zeta_1 & \zeta_2 & \zeta_3 & \dots  & \zeta_{n-k} \\
    \zeta_1^2 & \zeta_2^2 & \zeta_3^2 & \dots  & \zeta_{n-k}^2 \\
    \vdots & \vdots & \vdots &  & \vdots \\
    \zeta_1^{n-1} & \zeta_2^{n-1} & \zeta_3^{n-1} & \dots  & \zeta_{n-k}^{n-1}
\end{bmatrix} \quad ;$$
these are known to be the critical points of the Landau-Ginzburg
potential $W\in\mathcal{A}_{k,n}$ defined by Marsh-Rietsch \cite{MR}; see \cite[Theorem 1.1 and Corollary 3.12]{Ka}.
Observe that $S_{d^T}(I^\vee)=x_d(M_{I^\vee})/x_\emptyset(M_{I^\vee})$; this
follows from the expression of Schur polynomials as determinants \cite[Proposition 2.3 (1)]{Ca}.
After setting $x_\emptyset=1$, one can think of $W_\fraks$
as a regular function on the cluster chart $T_\fraks\subset\Gr^\vee(k,n)$
such that $W_\fraks=W_{|T_\fraks}$, as explained in (1) of
Proposition \ref{PropPluckerMutations}. This means that the critical points
of $W_\fraks$ are precisely those critical points $[M_{I^\vee}]\in\Gr^\vee(k,n)$
of $W$ such that $[M_{I^\vee}]\in T_\fraks$. By definition of $T_\fraks$
the latter condition is equivalent to $x_d(M_{I^\vee})\neq 0$ for all Young diagrams
$d$ appearing as labels of $Q_\fraks$, and thus $\xi_I$ is a well-defined local
system on $L_\fraks$ such that $HF((L_\fraks)_\xi,(L_\fraks)_\xi)\neq 0$ if and only if
$S_d(I)\neq 0$ for all $d$ appearing as labels of $Q_\fraks$.
\end{enumerate}
\end{proof}

\begin{lemma}\label{LemmaSimpleDyadicSpectrum}
If $n$ is odd, all the eigenvalues of $c_1\star$ acting on $\QH(\Gr(2,n))$ have
algebraic multiplicity one.
\end{lemma}

\begin{proof}
It was explained in part (2) of Proposition \ref{PropCritPointsValues} that the eigenvalues of $c_1\star$ acting on $\QH(\Gr(k,n))$
correspond to critical values of the Landau-Ginzburg potential $W$ on $\Gr^\vee(k,n)$
defined by Marsh-Rietsch \cite{MR}, and that the corresponding critical points
can be explicitly described. In particular, there are ${n \choose k}$ critical points,
and thus at most the same number of critical values. Therefore the statement is
equivalent to proving that there are precisely ${n \choose 2}$ distinct eigenvalues.
From part (1) of Proposition \ref{PropCritPointsValues}, each eigenvalue is of the form
$\lambda = n(\zeta_1 + \zeta_2)$, with $\zeta_1$ and $\zeta_2$ distinct roots
of $\zeta^n = -1$. Write $\zeta_1=e^{\frac{\pi i}{n} a}$ and $\zeta_1=e^{\frac{\pi i}{n} b}$ with
$0< a < b < 2n$ odd integers. The norm of one such eigenvalue is
$$|\lambda| = n\sqrt{2}\left( 1 + 2\cos\left(\frac{\pi}{n}(b-a)\right)\right)^{1/2} \quad .$$
The function $\cos(x)$ is decreasing for $0\leq x \leq \pi$ and $\cos(2\pi-x)=\cos(x)$;
in our case $0\leq \frac{\pi}{n}(b-a)\leq \pi$ whenever $0\leq b-a\leq n$. Since $n$ is odd by
assumption, by varying $a$ and $b$
among all odd integers with $0< a < b < 2n$ one finds $l=(n-1)/2$ eigenvalues
with $0<|\lambda_1|<\cdots <|\lambda_l|$, corresponding to $b-a$ attaining all
the even integer values in the interval $[2,n-1]$. Moreover, fixed any $1\leq t\leq l$,
the $n$ complex numbers $\lambda_t, (e^{\frac{2\pi}{n}i})\lambda_t, \ldots , (e^{\frac{2\pi}{n}i})^{n-1}\lambda_t$
are eigenvalues of $c_1\star$ too, and they have the same norm as $\lambda_t$;
see also \cite[Proposition 1.12]{Ca} for more on the symmetries of the spectrum of
$c_1\star$. Overall, we found $n(n-1)/2={n \choose 2}$ distinct eigenvalues.
\end{proof}

\begin{lemma}\label{LemmaForbiddenEdges}
Let $d\subseteq 2\times (n-2)$ be a Young diagram, and denote by $d^|=\{s,t\}$
with $1\leq s<t\leq n$ its vertical steps. Writing an arbitrary set $I\subset\{\zeta\in\CC \; : \; \zeta^n=-1\}$
with $|I|=2$ as $I_{a,b}=\{e^{\frac{\pi}{n}ia}, e^{\frac{\pi}{n}ib}\}$ with $0<a<b<2n$ odd integers, then
$$S_d(I_{a,b})= 0 \quad \iff \quad n \divides \frac{b-a}{2}(t-s) \quad .$$
\end{lemma}

\begin{proof}
Consider the full rank $n\times 2$ matrix
$$[M_{I_{a,b}}]=
\begin{bmatrix}
    1 & e^{\frac{\pi}{n}ia} & (e^{\frac{\pi}{n}ia})^2 & \dots  & (e^{\frac{\pi}{n}ia})^{n-1} \\
	1 & e^{\frac{\pi}{n}ib} & (e^{\frac{\pi}{n}ib})^2 & \dots  & (e^{\frac{\pi}{n}ib})^{n-1}
\end{bmatrix}^T$$
As pointed out in part (2) of Proposition \ref{PropCritPointsValues}, one has
$S_d(I_{a,b})=0$ if and only if $x_{d^T}(M_{I_{a,b}})=0$. Observe that the the
horizontal steps of the transpose diagram are $(d^T)^-=\{n+1-t, n+1-s\}$, so that
$$x_{d^T}(M_{I_{a,b}})= e^{\frac{\pi}{n}ia(n-t)}e^{\frac{\pi}{n}ib(n-s)}-e^{\frac{\pi}{n}ia(n-s)}e^{\frac{\pi}{n}ib(n-t)} \quad ,$$
from which $x_{d^T}(M_{I_{a,b}})=0$ if and only if $e^{\frac{\pi}{n}i(as+bt-at-bs)}=1$.
The last condition is verified precisely when $2n \divides (b-a)(t-s)$, and since
$b-a$ is the difference of two odd integers this can be rewritten as in the statement.
\end{proof}

\begin{theorem}\label{ThmDyadicHMS}
If $n=2^t+1$ for some $t\in\NN^+$, the derived Fukaya category $\D\Fuk(\Gr(2,2^t+1))$
is split-generated by objects supported on a single Pl\"{u}cker torus.
\end{theorem}

\begin{proof}
Up to replacing $2$ with $n-2$, we can think of the critical points $[M_{I_{a,b}}]$ of the Landau-Ginzburg
potential $W$ on $\Gr^\vee(2,n)$ defined by Marsh-Rietsch \cite{MR} as being parametrized
by sets $I_{a,b}=\{e^{\frac{\pi}{n}ia}, e^{\frac{\pi}{n}ib}\}$ with $0<a<b<2n$ odd integers;
compare with part (2) of Proposition \ref{PropCritPointsValues}. We claim that there exists
a Pl\"{u}cker sequence $\fraks$ of type $(2,n)$ such that the corresponding cluster
chart $T_\fraks\subset\Gr^\vee(2,n)$ contains all critical points $[M_{I_{a,b}}]$. If this
is true, then these will be also critical points of the Laurent polynomial $W_\fraks=W_{|T_\fraks}$,
which is the disk potential of the monotone Lagrangian torus $L_\fraks\subset\Gr(2,n)$
by Proposition \ref{PropSmallMonotone} and Lemma \ref{LemmaK2}. By Sheridan \cite[Corollary 2.19]{Sh},
if the generalized eigenspace $\QH_\lambda(X)$ of the operator $c_1\star$
is one-dimensional, any monotone Lagrangian brane $L_\xi$ with $HF(L_\xi,L_\xi)\neq 0$
split-generates $\D\Fuk_\lambda(X)$. Since $n=2^t+1$ is odd, by Lemma \ref{LemmaSimpleDyadicSpectrum}
we can apply this to $X=\Gr(2,n)$, $L=L_\fraks$ and any
$\xi=\xi_{I_{a,b}}$ for all $0<a<b<2n$ odd integers, thus concluding that the objects
$(L_\fraks)_{I_{a,b}}$ split-generate every summand of $\D\Fuk(\Gr(2,2^t+1))$. The construction
of the Pl\"{u}cker sequence $\fraks$ mentioned in the claim above proceeds as follows.

\begin{figure}[H]
  \centering
        \includegraphics[width=0.5\textwidth]{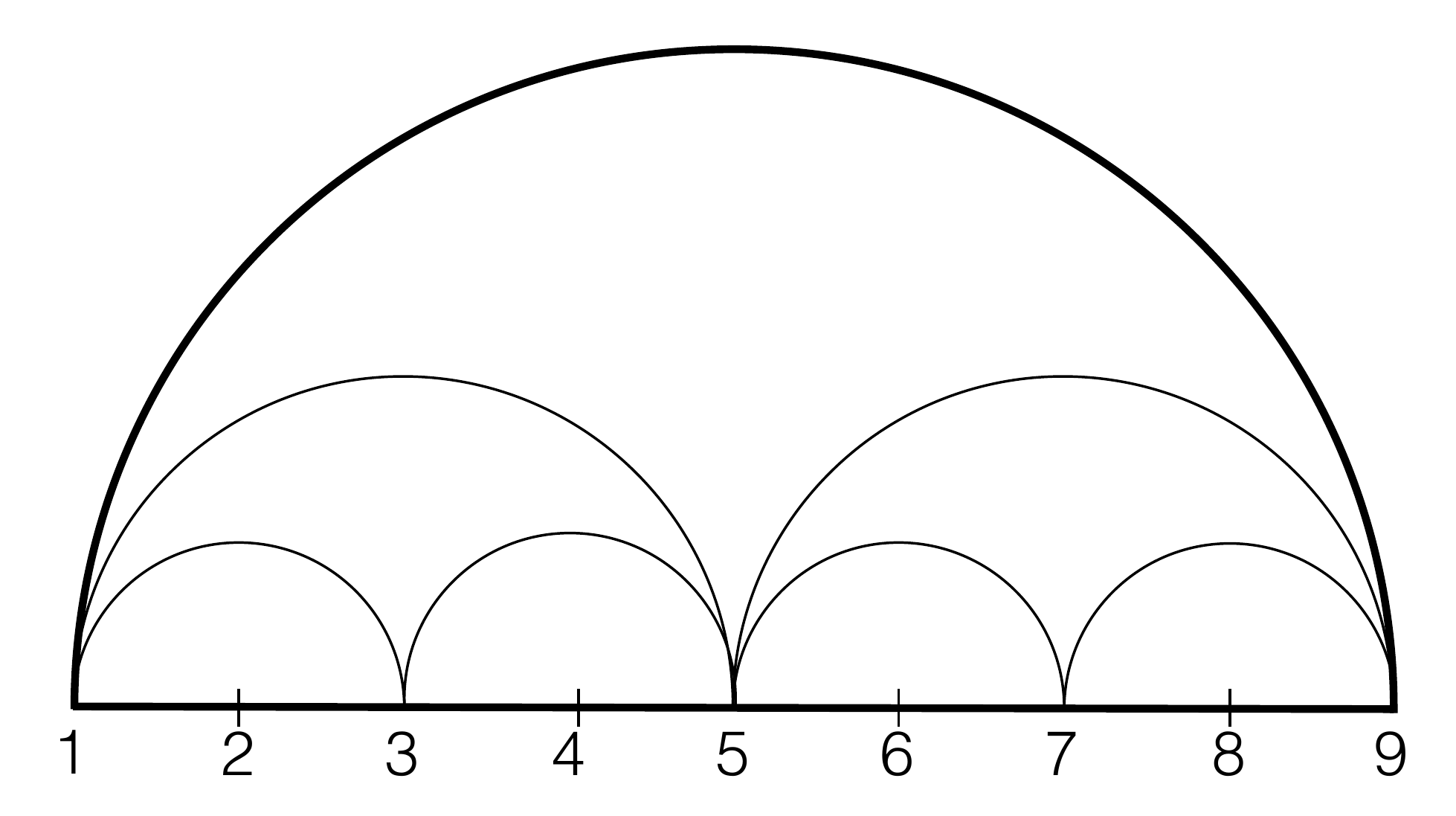}
    \caption{The dyadic triangulation of an $9$-gon.}
    \label{FigDyadicGr29}
\end{figure}

Consider the following incremental construction of a set $\Gamma$ (an example with $t=3$ is given
in Figure \ref{FigDyadicGr29}):

\begin{enumerate}
	\item start with a segment partitioned in $n-1=2^t$ intervals, which are added to $\Gamma$ as new edges
		$\{1,2\}$, $\{2,3\}$, \ldots , $\{2^t,2^t+1\}$ ;
	\item partition the segment into $(n-1)/2=2^{t-1}$ pairs of consecutive intervals,
		and add a new arc connecting the left end of the left interval to the right end
		of the right interval in each pair, thus adding new edges $\{1,3\}$,
		$\{3,5\}$, \ldots ,$\{2^{t-1}, 2^{t+1}\}$ to $\Gamma$ ;
	\item partition the segment in $(n-1)/2^2=2^{t-2}$ tuples of $2^2$ consecutive intervals,
		and add a new arc connecting the left end of the leftmost interval to the right end
		of the rightmost interval in each tuple, thus adding new edges $\{1,5\}$,
		$\{5,9\}$, \ldots ,$\{2^{t+1}-2^2, 2^{t+1}\}$ to $\Gamma$ ;
	\item proceed as above until the initial segment is partitioned in two tuples
		of $2^{t-1}$ consecutive intervals, and add the edge $\{1,n\}=\{1,2^t+1\}$ to $\Gamma$,
		so that it becomes a triangulation of $[n]$ in the sense of Definition \ref{DefTriangulation} .
\end{enumerate}

Let $(Q_0,W_0)$ be the initial seed of Definition \ref{DefInitialSeed}, and
call $\Gamma_0$ the triangulation of $[n]$ corresponding to Young diagrams
labeling the nodes of $Q_0$ as in Lemma \ref{LemmaTriangClusters}. The triangulation
$\Gamma_0$ is connected to $\Gamma$ constructed above by a sequence of flips,
which correspond to mutations of the quiver $Q_0$ at nodes with to incoming and
two outgoing arrows. From Proposition \ref{PropPluckerMutations}, this gives
a Pl\"{u}cker sequence of mutations of type $(2,n)$ in the sense of Definition
\ref{DefPluckerSequence}, which ends at $(Q_\fraks,W_\fraks)$ and such that the labels
of $Q_\fraks$ correspond to the triangulation $\Gamma_\fraks=\Gamma$, again
by Lemma \ref{LemmaTriangClusters}. It remains to show that $[M_{I_{a,b}}]\in T_\fraks$
for all odd integers $a$ and $b$ such that $0<a<b<2n$. Suppose not, then there
exist some $a,b$ and some Young diagram $d\subseteq 2\times (n-2)$ such that $x_d(M_{I_{a,b}})=0$.
By Lemma \ref{LemmaForbiddenEdges}, this implies that $n \divides \frac{b-a}{2}(t-s)$,
where $d^|=\{s,t\}$ are the vertical steps of $d$. By construction, for any $d^|\in\Gamma_\fraks$,
if $d^|=\{s,t\}$ then $t-s$ is a power of $2$, and since $n=2^t+1$ is odd by assumption
we must have $n\divides \frac{b-a}{2}$. This is impossible, because $\frac{b-a}{2}< n$.
\end{proof}

\begin{example}
$\D\Fuk(\Gr(2,9))$ is generated by a single Pl\"{u}cker
torus. Note that instead the Gelfand-Cetlin torus mentioned in Section \ref{SecIntro}
does not support enough nonzero objects to generate; compare \cite[Figure 2(C)]{Ca}.
\end{example}

The arguments above can be generalized to prove that certain collections
of Pl\"{u}cker tori split generate $\D\Fuk(\Gr(2,n))$.

\begin{definition}\label{DefPAvoiding}
Let $p$ be a prime number. A triangulation $\Gamma$ of $[n]$ as in Definition \ref{DefTriangulation}
is called $p$-avoiding if for all $\{s,t\}\in\Gamma$ one has $p \nmid (t-s)$.
\end{definition}

\begin{theorem}\label{ThmPavoidingHMS}
Let $n>2$ be odd, and consider its prime factorization $n=p_1^{e_1}\cdots p_l^{e_l}$.
Assume that for all $1\leq i\leq l$ there exists a triangulation $\Gamma_i$ of $[n]$
that is $p_i$-avoiding, then $\D\Fuk(\Gr(2,n))$ is split generated by objects
supported on $l$ Pl\"{u}cker tori.
\end{theorem}

\begin{proof}
Recall that up to replacing $2$ with $n-2$, we can think of the critical points $[M_{I_{a,b}}]$ of the Landau-Ginzburg
potential $W$ on $\Gr^\vee(2,n)$ defined by Marsh-Rietsch \cite{MR} as being parametrized
by sets $I_{a,b}=\{e^{\frac{\pi}{n}ia}, e^{\frac{\pi}{n}ib}\}$ with $0<a<b<2n$ odd integers;
compare with part (2) of Proposition \ref{PropCritPointsValues}. Denote $\mathcal{C}$
the set of all critical points of $W$, and for $1\leq i\leq l$ define
$$ \mathcal{C}_{p_i} = \{ [M_{I_{a,b}}]\in\mathcal{C} \; : \; p_i^{e_i} \nmid \frac{b-a}{2} \} \quad .$$
Observe that $\mathcal{C}=\mathcal{C}_{p_1}\cup\cdots\cup\mathcal{C}_{p_l}$. Indeed,
if $p_i^{e_i}\divides (b-a)/2$ for all $1\leq i \leq l$ then $p_1^{e_1}\cdots p_l^{e_l}=n\divides (b-a)/2$,
against the fact that $(b-a)/2 < n$. By assumption, for each $1\leq i\leq l$ there
exist a triangulation $\Gamma_i$ of $[n]$ that is $p_i$-avoiding, and arguing as
in Theorem \ref{ThmDyadicHMS} one finds a Pl\"{u}cker sequence $\fraks_i$ of type
$(2,n)$ that starts with the initial seed $(Q_0,W_0)$ and ends with $(Q_{\fraks_i},W_{\fraks_i})$,
and such that the labels of $Q_{\fraks_i}$ correspond to the triangulation $\Gamma_{\fraks_i}=\Gamma_i$
as in Lemma \ref{LemmaTriangClusters}. Each of the $l$ Pl\"{u}cker tori $L_{\fraks_i}\subset\Gr(2,n)$
has an associated cluster chart $T_{\fraks_i}\subset\Gr^\vee(2,n)$, and we claim
that $\mathcal{C}_{p_i}\subset T_{\fraks_i}$. Suppose not, then there exists some
$[M_{I_{a,b}}]\in \mathcal{C}_{p_i}$ such that $[M_{I_{a,b}}]\notin T_{\fraks_i}$.
This means that $p_i^{e_i} \nmid \frac{b-a}{2}$ and there exists some Young diagram
$d\subseteq 2\times (n-2)$ such that $x_d(M_{I_{a,b}})=0$, and denoting $d^|=\{s,t\}$
its vertical steps $\{s,t\}\in\Gamma_{\fraks_i}$. By Lemma \ref{LemmaForbiddenEdges}, this implies
that $n \divides \frac{b-a}{2}(t-s)$, and so in particular $p_i^{e_i} \divides \frac{b-a}{2}(t-s)$.
Since $\Gamma_{\fraks_i}$ is $p_i$-avoiding, this means that $p_i^{e_i} \divides \frac{b-a}{2}$,
against the fact that $[M_{I_{a,b}}]\in \mathcal{C}_{p_i}$. As in Theorem \ref{ThmDyadicHMS},
the assumption $n$ odd and Lemma \ref{LemmaSimpleDyadicSpectrum} guarantee,
by Sheridan \cite[Corollary 2.19]{Sh}, that any nonzero object of the Fukaya
category supported on one of the $l$ monotone Pl\"{u}cker tori $L_{\fraks_1},\ldots ,L_{\fraks_l}\subset\Gr(2,n)$
split-generates the summand $\D\Fuk_{\lambda}(\Gr(2,n))$ of the derived
Fukaya category containing it. The objects supported
on $L_{\fraks_i}$ are obtained by endowing it with local systems $\xi_{I_{a,b}}$
as in Definition \ref{DefLocalSystems}
corresponding to critical points $[M_{I_{a,b}}]\in T_{\fraks_i}$; these are
such that $HF((L_{\fraks_i})_{\xi_{I_{a,b}}}, (L_{\fraks_i})_{\xi_{I_{a,b}}})\neq 0$ because the
disk potential of $L_{\fraks_i}$ is $W_{\fraks_i}=W_{|T_{\fraks_i}}$.
\end{proof}

\begin{figure}[H]
  \centering
    \begin{subfigure}[b]{0.4\textwidth}
        \includegraphics[width=\textwidth]{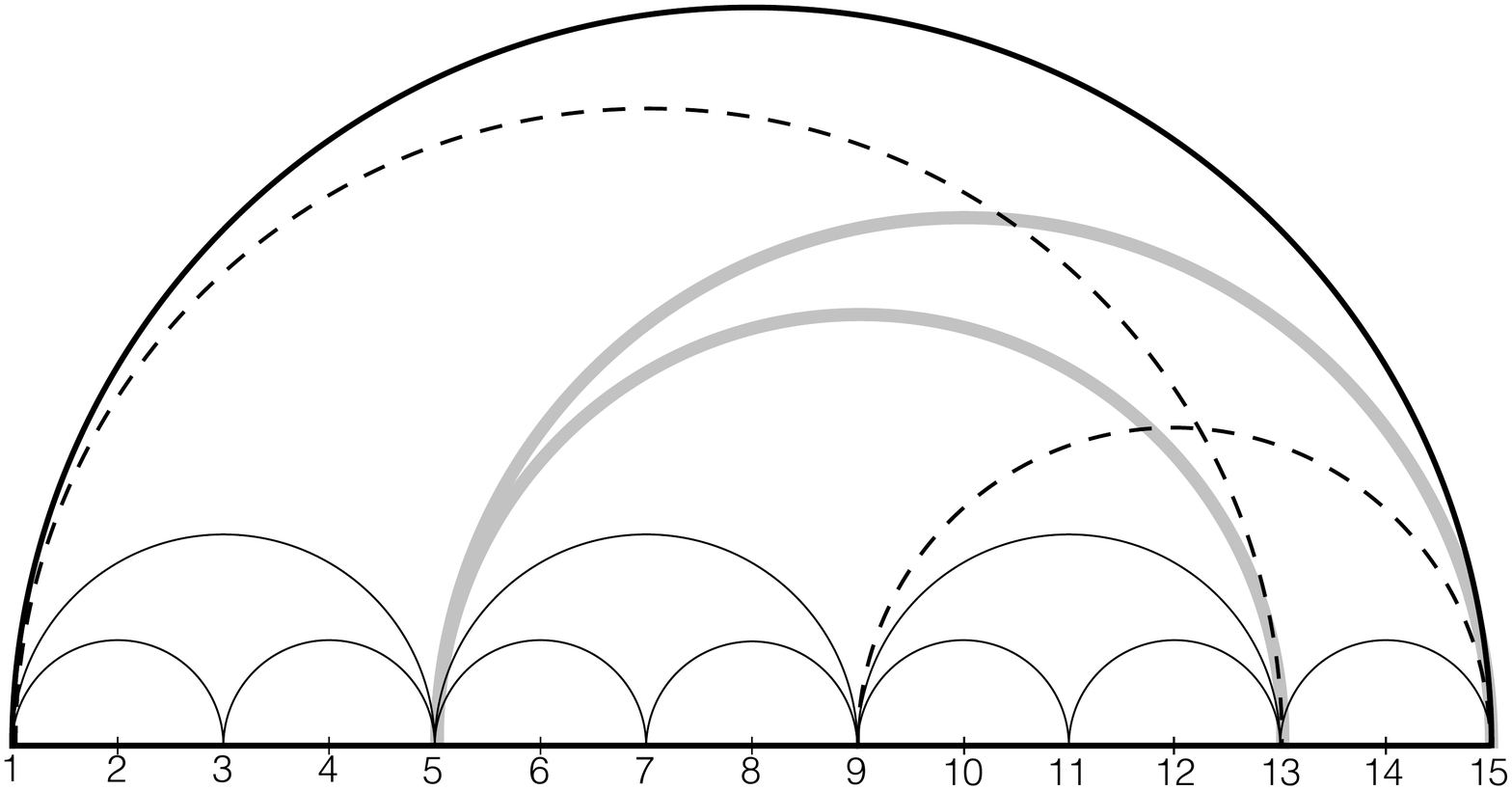}
        \caption{$3$-avoiding triangulation}
    \end{subfigure}
    \begin{subfigure}[b]{0.4\textwidth}
        \includegraphics[width=\textwidth]{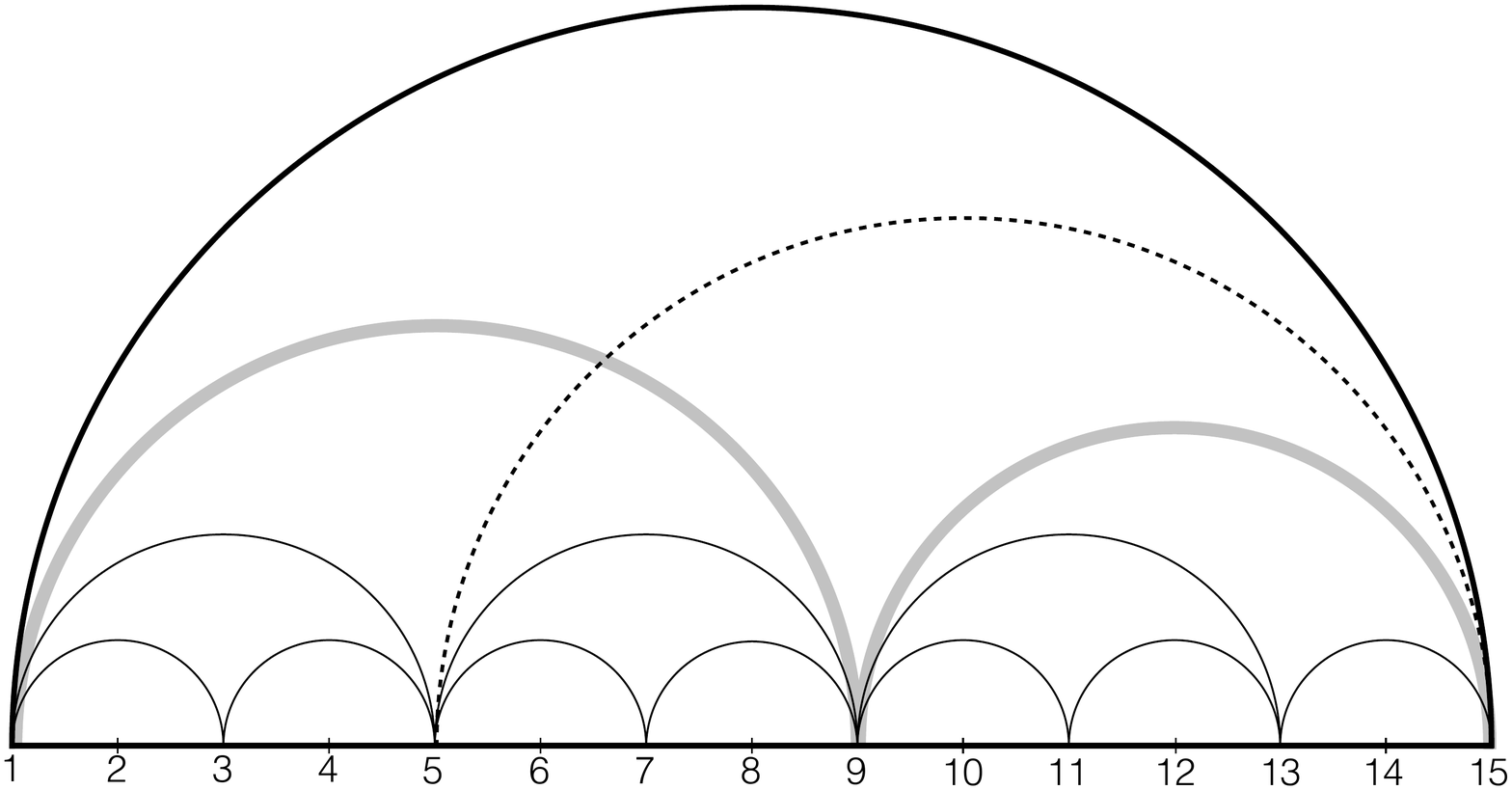}
        \caption{$5$-avoiding triangulation}
    \end{subfigure}

    \caption{Two triangulations of a $15$-gon.}
    \label{FigPluckerSeq25}
\end{figure}

\begin{example}
$\D\Fuk(\Gr(2,15))$ is generated by two Pl\"{u}cker tori, whose corresponding
triangulations are shown in Figure \ref{FigPluckerSeq25}. To get the two triangulations,
one starts by constructing a partial triangulation of the $15$-gon with dyadic arcs as in Theorem
\ref{ThmDyadicHMS} (solid arcs in Figure \ref{FigPluckerSeq25}). The partial triangulation is $p$-avoiding for every prime $p>2$
by construction. Since $15=3\cdot 5$, by Theorem \ref{ThmPavoidingHMS} one needs
to find completions of the partial triangulation to full triangulations that
are $3$-avoid and $5$-avoiding respectively. In Figure \ref{FigPluckerSeq25},
the remaining arcs $\{i,j\}$ with $3\divides (j-i)$ are coarsely dashed, while the
one with $5\divides (j-i)$ is finely dashed; triangulation (A) is obtained by
adding two shaded arcs and is $3$-avoiding, while triangulation (B) is obtained by
adding two different shaded arcs and is $5$-avoiding.
\end{example}

\subsection{Exotic Lagrangian tori in $\Gr(3,6)$}

\begin{definition}\label{DefFVector}
If $L_\fraks\subset\Gr(k,n)$ is a Pl\"{u}cker Lagrangian of type $(k,n)$,
define its $f$-vector to be
$$\mathbf{f}(L_\fraks)=(f_1,\ldots ,f_{k(n-k)})\in\NN^{k(n-k)} \quad ,$$
where $f_i$ is the number of $(i-1)$-dimensional faces in the Newton polytope $P_\fraks$
of the potential $W_\fraks$.
\end{definition}

\begin{definition}\label{DefWeight}
If $L_\fraks\subset\Gr(k,n)$ is a Pl\"{u}cker Lagrangian of type $(k,n)$,
define its weight $\mathbf{wt}(L_\fraks)\in\NN$ to be the number of sets
$$I\subset\{ \zeta \in \CC \; : \; \zeta^n = (-1)^{k+1} \}$$
such that $|I|=k$ and $S_d(I)\neq 0$ for all Young diagrams $d$ appearing as
labels of the quiver $Q_\fraks$.
\end{definition}

\begin{lemma}\label{LemmaHamInvariance}
Assume $\fraks, \fraks'$ are Pl\"{u}cker sequences of type $(k,n)$ satisfying Conjecture
\ref{ConjPotentials}. If $\mathbf{f}(L_\fraks)\neq\mathbf{f}(L_{\fraks'})$ or
$\mathbf{wt}(L_\fraks)\neq\mathbf{wt}(L_{\fraks'})$, then the Lagrangian tori
$L_\fraks,L_{\fraks'}\subset\Gr(k,n)$ are not Hamiltonian isotopic.
\end{lemma}

\begin{proof}
Suppose that there exists a Hamiltonian isotopy $\phi$ such
that $\phi(L_\fraks)=L_{\fraks'}$. Then by assumption the induced map $\phi_*: H_1(L_\fraks;\ZZ)\to H_1(L_{\fraks'};\ZZ)$
is such that
$$W_{\fraks} \sim W_{\fraks'} \circ (\phi_* \otimes \operatorname{id}_{\CC^\times}) \quad ,$$
where $\sim$ denotes equality up to automorphisms of $T_\fraks$.
This means that the Newton polytopes $P_\fraks$ and $P_{\fraks'}$ of the Laurent
polynomials $W_\fraks$ and $W_{\fraks'}$ are related by a transformation of
$GL(k(n-k),\ZZ)$, and hence have the same $f$-vector, because the number of faces
of any given dimension of a polytope is a unimodular invariant; this proves that
$\mathbf{f}(L_\fraks)=\mathbf{f}(L_{\fraks'})$. Moreover, the Laurent
polynomials $W_\fraks$ and $W_{\fraks'}$ can be thought of as regular functions
on a torus $(\CC^\times)^{k(n-k)}$, which agree up to an automorphism. Since
the number of critical points of a function is invariant under automorphisms
of its domain, it follows from part (2) of Proposition \ref{PropCritPointsValues}
that $\mathbf{wt}(L_\fraks)=\mathbf{wt}(L_{\fraks'})$.
\end{proof}

\begin{table}[H]\label{Table36}
\resizebox{\textwidth}{!}{%
\begin{tabular}{ cccc  }
	\hline
	\multicolumn{4}{|c|}{$k=3$, $n=6$} \\
	\hline
 	$L_\fraks$ &	Labels of $Q_\fraks$ &	$\mathbf{f}(L_\fraks)$	& $\mathbf{wt}(L_\fraks)$ \\
 	\hline
	1	&	123, 124, 125, 126, 156, 234, 245, 256, 345, 456	&	(14, 83, 276, 571, 766, 670, 372, 122, 20)	&	18\\
	2	&	123, 124, 125, 126, 145, 156, 234, 245, 345, 456	&	(14, 83, 276, 571, 766, 670, 372, 122, 20)	&	18\\
	3	&	123, 125, 126, 135, 145, 156, 234, 235, 345, 456	&	(15, 91, 302, 615, 807, 690, 376, 122, 20)	&	6\\
	4	&	123, 126, 134, 136, 146, 156, 234, 345, 346, 456	&	(14, 83, 276, 571, 766, 670, 372, 122, 20)	&	18\\
	5	&	123, 126, 156, 234, 235, 236, 245, 256, 345, 456	&	(15, 91, 302, 615, 807, 690, 376, 122, 20)	&	6\\
	6	&	123, 125, 126, 156, 234, 235, 245, 256, 345, 456	&	(14, 83, 276, 571, 766, 670, 372, 122, 20)	&	18\\
	7	&	123, 124, 125, 126, 134, 145, 156, 234, 345, 456	&	(15, 91, 302, 615, 807, 690, 376, 122, 20)	&	18\\
	8	&	123, 125, 126, 134, 135, 145, 156, 234, 345, 456	&	(16, 98, 322, 645, 832, 701, 378, 122, 20)	&	6\\
	9	&	123, 124, 126, 146, 156, 234, 245, 246, 345, 456	&	(16, 98, 322, 645, 832, 701, 378, 122, 20)	&	6\\
	10	&	123, 126, 156, 234, 236, 246, 256, 345, 346, 456	&	(15, 91, 302, 615, 807, 690, 376, 122, 20)	&	6\\
	11	&	123, 124, 126, 146, 156, 234, 246, 345, 346, 456	&	(15, 91, 302, 615, 807, 690, 376, 122, 20)	&	6\\
	12	&	123, 124, 126, 145, 146, 156, 234, 245, 345, 456	&	(15, 91, 302, 615, 807, 690, 376, 122, 20)	&	18\\
	13	&	123, 126, 146, 156, 234, 236, 246, 345, 346, 456	&	(16, 98, 322, 645, 832, 701, 378, 122, 20)	&	6\\
	14	&	123, 126, 156, 234, 236, 256, 345, 346, 356, 456	&	(14, 83, 276, 571, 766, 670, 372, 122, 20)	&	18\\
	15	&	123, 126, 146, 156, 234, 236, 245, 246, 345, 456	&	(18, 111, 358, 700, 882, 728, 386, 123, 20)	&	6\\
	16	&	123, 126, 156, 234, 235, 236, 256, 345, 356, 456	&	(14, 83, 276, 571, 766, 670, 372, 122, 20)	&	18\\
	17	&	123, 124, 126, 134, 145, 146, 156, 234, 345, 456	&	(14, 83, 276, 571, 766, 670, 372, 122, 20)	&	18\\
	18	&	123, 126, 134, 135, 136, 156, 234, 345, 356, 456	&	(16, 98, 322, 645, 832, 701, 378, 122, 20)	&	6\\
	19	&	123, 126, 134, 136, 145, 146, 156, 234, 345, 456	&	(14, 83, 276, 571, 766, 670, 372, 122, 20)	&	18\\
	\rowcolor{lightgray} 20	&	123, 126, 135, 136, 145, 156, 234, 235, 345, 456	&	(15, 93, 317, 661, 882, 760, 413, 132, 21)	&	6\\
	21	&	123, 126, 136, 146, 156, 234, 236, 345, 346, 456	&	(15, 91, 302, 615, 807, 690, 376, 122, 20)	&	18\\
	22	&	123, 125, 126, 134, 135, 156, 234, 345, 356, 456	&	(18, 111, 358, 700, 882, 728, 386, 123, 20)	&	6\\
	23	&	123, 126, 136, 156, 234, 235, 236, 345, 356, 456	&	(14, 83, 276, 571, 766, 670, 372, 122, 20)	&	18\\
	24	&	123, 126, 134, 135, 136, 145, 156, 234, 345, 456	&	(15, 91, 302, 615, 807, 690, 376, 122, 20)	&	6\\
	25	&	123, 124, 126, 156, 234, 245, 246, 256, 345, 456	&	(15, 91, 302, 615, 807, 690, 376, 122, 20)	&	6\\
	26	&	123, 126, 134, 136, 156, 234, 345, 346, 356, 456	&	(15, 91, 302, 615, 807, 690, 376, 122, 20)	&	18\\
	27	&	123, 126, 135, 136, 156, 234, 235, 345, 356, 456	&	(15, 91, 302, 615, 807, 690, 376, 122, 20)	&	6\\
	28	&	123, 124, 126, 134, 146, 156, 234, 345, 346, 456	&	(14, 83, 276, 571, 766, 670, 372, 122, 20)	&	18\\
	29	&	123, 126, 156, 234, 236, 245, 246, 256, 345, 456	&	(16, 98, 322, 645, 832, 701, 378, 122, 20)	&	18\\
	30	&	123, 126, 136, 156, 234, 236, 345, 346, 356, 456	&	(14, 83, 276, 571, 766, 670, 372, 122, 20)	&	18\\
	31	&	123, 125, 126, 145, 156, 234, 235, 245, 345, 456	&	(14, 83, 276, 571, 766, 670, 372, 122, 20)	&	18\\
	32	&	123, 125, 126, 135, 156, 234, 235, 345, 356, 456	&	(16, 98, 322, 645, 832, 701, 378, 122, 20)	&	6\\
	33	&	123, 125, 126, 156, 234, 235, 256, 345, 356, 456	&	(15, 91, 302, 615, 807, 690, 376, 122, 20)	&	18\\
	\rowcolor{lightgray} 34	&	123, 124, 126, 156, 234, 246, 256, 345, 346, 456	&	(15, 93, 317, 661, 882, 760, 413, 132, 21)	&	6\\
	\hline
\end{tabular}}
\end{table}

\begin{theorem}\label{ThmExoticTori}
The Grassmannian $\Gr(3,6)$ contains at least $6$ monotone Lagrangian tori
that are non-displaceable and pairwise inequivalent under Hamiltonian isotopy.
\end{theorem}

\begin{proof}
The table above contains informations about the steps of a Pl\"{u}cker
sequence $\fraks$ of type $(3,6)$. In each row, the reader can find the Young
diagrams $d\subseteq 3\times 3$ appearing as labels of $Q_\fraks$ at a given step,
identified by their sets of vertical sets $\{i,j,k\}\subset [6]$. Each potential
$W_\fraks$ has an associated Newton polytope $P_\fraks$, whose $f$-vector is
$\mathbf{f}(L_\fraks)$ as in Definition \ref{DefFVector}. Following Definition
\ref{DefWeight}, the weight $\mathbf{w}(L_\fraks)$ is computed by counting
how many of the ${6 \choose 3}$ sets $I$ of roots of $\zeta^6=1$
with $|I|=3$ have the property that $S_d(I)\neq 0$ for all Young diagrams
$d\subseteq 3\times 3$ that appear as labels on the nodes of the quiver $Q_\fraks$.
Calling $\Sigma_\fraks=\Sigma^fP_\fraks$ the face fan of the Newton polytope,
by Proposition \ref{PropSmallMonotone} the Lagrangian torus $L_\fraks\subset\Gr(k,n)$
is monotone and has disk potential $W_\fraks$ whenever the toric variety
$X(\Sigma_\fraks)$ has a small toric resolution in the sense of Definition
\ref{DefSmallResolution}. This condition can be checked algorithmically at
each step, since every fan has finitely many simplicial refinements with the same
rays, and every smooth refinement is in particular simplicial. For the
$34$ steps in the table, the code \cite{CaCode} finds small resolutions in $32$
cases; the remaining $2$ cases are marked gray in the table (we did not actually
check all possible simplicial refinements in these cases, so small toric resolutions
for them may still exist). From Lemma \ref{LemmaHamInvariance},
we conclude that $\Gr(3,6)$ contains at least $6$ monotone Lagrangian tori
that are pairwise not Hamiltonian isotopic: rows 1, 3, 7, 8, 15, 29.
Regarding nondisplaceability, it suffices to show that the $32$ tori $L_\fraks\subset\Gr(3,6)$
have Floer cohomology $HF(L_\xi, L_\xi)\neq 0$ for some local
system $\xi$. By Auroux \cite[Proposition 6.9]{Au} and Sheridan
\cite[Proposition 4.2]{Sh}, the Floer cohomology of a monotone
Lagrangian torus brane $L_\xi$ is nonzero if and only if the holonomy $\hol_\xi$ of
its local system $\xi$ is a critical point of the disk potential disk potential
$W_\fraks$. Therefore, it suffices to show that each of the $32$ Laurent
polynomials $W_\fraks$ has at least one critical point. Thinking $W_\fraks$
as restriction $W_\fraks=W_{|T_\fraks}$ of the Landau-Ginzburg potential $W$
on $\Gr^\vee(3,6)$ defined by Marsh-Rietsch \cite{MR} to the cluster chart
$T_\fraks\subset\Gr^\vee(3,6)$, it suffices to show that each of the charts contains
at least one critical point of $W$. In fact, something stronger is true: there is a critical point of $W$
that is contained in $T_\fraks$ for all $\fraks$. As proved by Rietsch \cite{Ri06} (see also Karp
\cite{Ka}), for any $1\leq k<n$ there is a (unique) critical point of $W$ in the totally positive part
$\Gr^\vee(k,n)_{>0}\subset\Gr^\vee(k,n)$, i.e. the locus where all Pl\"{u}cker coordinates
are real and positive. Following the notation of part (2) in Proposition \ref{PropCritPointsValues},
this point is $[M_{I_0}]\in\Gr^\vee(k,n)$ with $I_0$ the set of $k$ roots
of $\zeta^n=(-1)^{k+1}$ closest to $1$. Applying this to $(k,n)=(3,6)$, and
recalling that $[M_{I_0}]\in T_\fraks$
if and only if $x_d(M_{I_0})$ for all Young diagrams $d\subseteq 3\times 3$ appearing
as labels on the nodes of $Q_\fraks$, we conclude that the the total positivity of $[M_{I_0}]$
implies that it belongs to every cluster chart $T_\fraks$, and this proves that
all $L_\fraks$ are nondisplaceable.
\end{proof}

\begin{remark}
We emphasize that the arguments of Theorem \ref{ThmExoticTori} prove that any
$L_\fraks\subset\Gr(k,n)$ is nondisplaceable as long as $W_\fraks = W_{L_\fraks}$.
This is due to the fact that the tori $L_\fraks$ correspond to cluster
charts $T_\fraks\subset\Gr^\vee(k,n)$ by construction, and that $W$ has a critical point in the intersection
of all such charts.
\end{remark}

\begin{remark}
It was shown in \cite[Theorem 4.8]{Ca} that the dihedral group
$$D_n = \langle \; r,s \; | \; r^n=s^2=1, \; rs = sr^{-1} \; \rangle$$
acts on the set of critical points of $W\in\mathcal{A}_{k,n}$, and that the cluster chart
$T_\fraks$ is invariant under the action of the subgroup $\langle r\rangle$. Since
$\mathbf{wt}(L_\fraks)$ is the number of critical points in $T_\fraks$, the fact
that $\ZZ /n\ZZ$ acts on it puts some arithmetic constraints on this number.
\end{remark}

\bibliographystyle{abbrv}
\bibliography{biblio}

\end{document}